\documentclass[a4paper,11pt]{siamart1116}
\usepackage{LatexDefinitions}

\title{Stabilized Sparse Scaling Algorithms for Entropy Regularized Transport Problems}
\author{Bernhard Schmitzer}
\date{\today}
\hypersetup{  
   pdftitle={Stabilized Sparse Scaling Algorithms for Entropy Regularized Transport Problems},
   pdfauthor={Bernhard Schmitzer}
}

\newcommand{\prob}{\mc{P}}

\begin{document}
\maketitle


\begin{abstract}
Scaling algorithms for entropic transport-type problems have become a very popular numerical method, encompassing Wasserstein barycenters, multi-marginal problems, gradient flows and unbalanced transport.
However, a standard implementation of the scaling algorithm has several numerical limitations: the scaling factors diverge and convergence becomes impractically slow as the entropy regularization approaches zero. Moreover, handling the dense kernel matrix becomes unfeasible for large problems.
To address this, we combine several modifications: A log-domain stabilized formulation, the well-known $\veps$-scaling heuristic, an adaptive truncation of the kernel and a coarse-to-fine scheme.
This permits the solution of larger problems with smaller regularization and negligible truncation error.
A new convergence analysis of the Sinkhorn algorithm is developed, working towards a better understanding of $\veps$-scaling.
Numerical examples illustrate efficiency and versatility of the modified algorithm.
\end{abstract}

\tableofcontents


\section{Introduction}
\label{sec:Introduction}

\subsection{Motivation and Related Work}
\label{sec:IntroductionRelatedWork}

\paragraph{Applications of Optimal Transport}
Optimal transport (OT) is a classical optimization problem dating back to the seminal work of Monge and Kantorovich (see monographs \cite{Villani-OptimalTransport-09,Santambrogio-OTAM} for introduction and historical context).
The induced Wasserstein distances lift a metric from a `base' space $(X,d)$ to probability measures over $X$. This is a powerful analytical tool, for example to study PDEs as gradient flows in Wasserstein space \cite{JKO1998,AmbrosioGradientFlows2005}.
With the increase of computational resources, OT has also become a popular numerical tool in image processing, computer vision and machine learning (e.g.~\cite{RubnerEMD-IJCV2000,SchmitzerSchnoerr-WassersteinModes2014,RumpfGeneralizedOT2014,CuturiGroundMetric2014,Fitschen-ColorTransport-JMIV2016}).

Many ideas have been presented to extend Wasserstein distances to general non-negative measures. We refer to \cite{KMV-OTFisherRao-2015,ChizatOTFR2015,LieroMielkeSavare-HellingerKantorovich-2015a,ChizatDynamicStatic2018} and references therein for some context.
A transport-type distance for general multi-channel signals is proposed in \cite{SlepcevTLP2016}.

\paragraph{Computational Optimal Transport} 
To this day, the computational effort to solve OT problems remains the principal bottleneck in many applications. In particular large problems, or even multi-marginal problems, remain challenging both in terms of runtime and memory demand.

For the linear assignment problem and discrete transport problems there are (combinatorial) algorithms based on the finite dimensional linear programming formulation by Kantorovich, such as the Hungarian method \cite{KuhnHungarianMethod}, the auction algorithm \cite{Bertsekas-ParallelAuction1988}, the network simplex \cite{NetworkFlows1993} and more \cite{GoldbergCostScaling1990}.
Typically, they work for (almost) arbitrary cost functions, but do not scale well for large, dense problems.
On the other hand, there are more geometric solvers, relying on the polar decomposition \cite{MonotoneRerrangement-91},
that tend to be more efficient. There is the famous fluid dynamic formulation by Benamou and Brenier \cite{BenamouBrenier2000},
explicit computation of the polar decomposition \cite{OptimalTransportWarping}, semi-discrete solvers \cite{MultiscaleTransport2011,LevySemiDiscrete2015}, and solvers of the Monge-Amp\`ere equation \cite{ObermanMongeAmpere2014,BenamouCollinoMirebeau_MongeAmpere2014} among many others. However, these only work on very specific cost functions, most notably the squared Euclidean distance.
In a compromise between efficiency and flexibility, several discrete coarse-to-fine solvers have been proposed that adaptively select sparse sub-problems \cite{SchmitzerSchnoerr-SSVM2013,ObermanOptimalTransportationLP2015,SchmitzerShortCuts2015}.

\paragraph{Entropy Regularization for Optimal Transport}
In \cite{YuilleInvisibleHand1994} entropy regularization of the linear assignment problem is considered to allow application of smooth optimization techniques or the Sinkhorn matrix scaling algorithm \cite{SinkhornKnopp1967}.
For sufficiently small regularization the true optimal assignment can be extracted from the approximate solution.
For increased numerical stability, the Sinkhorn algorithm is also reformulated in the log-domain.
Similarly, in \cite{Cuturi2013} the Sinkhorn algorithm is applied to solve an entropy regularized approximation of the discrete optimal transport problem.
It is demonstrated that for moderate regularization strengths the algorithm is trivial to parallelize, easy to implement on GPUs and fast. Besides, it is shown that moderate regularization can actually be beneficial for classification applications. Regularization also makes the optimization problem more well-behaved (e.g.~uniqueness of optimal coupling, optimal objective differentiable as function of marginal distributions), which led to the first practical numerical method for approximate computation of Wasserstein barycenters \cite{Cuturi14Barycenters}.
Today, this approach is widely used, for instance \cite{Solomon-siggraph-2015,RabinPapadakisSSVM2015,SlepcevTLP2016,MandadSurfaceMatching2017}.

More recently, the Sinkhorn algorithm has been extended to more general transport-type problems, such as multi-marginal problems and direct computation of Wasserstein barycenters \cite{BenamouIterativeBregman2015}, gradient flows \cite{PeyreEntropicFlows2015} and unbalanced transport problems \cite{ChizatEntropicNumeric2018}, resulting in a family of Sinkhorn-like diagonal scaling algorithms.

Convergence of the discrete regularized problem towards the unregularized limit is studied in \cite{Cominetti-ExpBarrierConvergence-1992}. In a continuous setting, this is related to the Schr\"odinger problem and the lazy gas experiment (see \cite{LeonardSchroedingerMK2012} for a review and a very general convergence proof).
\cite{Carlier-EntropyJKO-2015} provides a simpler and direct analysis for the 2-Wasserstein distance on $\R^d$ and studies the limit of entropy regularized gradient flows.

\paragraph{Convergence Speed of Sinkhorn Algorithm}

In \cite{FranklinLorenz-Scaling-1989} the convergence rate of the Sinkhorn algorithm is studied for positive kernel matrices, yielding a global linear convergence rate of the marginals in terms of Hilbert's projective metric.
However, applied to entropy regularized optimal transport, the contraction factor tends to one exponentially, as the regularization approaches zero. Thus, running the algorithm with this particular measure of convergence is often not practically feasible.
In \cite{Knight-SinkhornKnopp} the local convergence rate of the Sinkhorn algorithm near the solution is examined, based on a linearization of the iterations.
This bound is tighter and more accurately describes the behaviour of the algorithm close to convergence. But these estimates do not apply when one starts far from the optimal solution, which is the usual case for small regularization parameters.
In \cite{YuilleInvisibleHand1994} a comparison is made between the Sinkhorn algorithm and the auction algorithm.
In particular the role of the entropy regularization parameter is related to the slack parameter $\veps$ of the auction algorithm and it is pointed out that convergence of both algorithms becomes slower, as these parameters approach zero (but small parameters are required for good approximate solutions).
For the auction algorithm this can provably be remedied by $\veps$-scaling, where the $\veps$ parameter is gradually decreased during optimization.
Analogously, it is suggested to gradually decrease entropy regularization during the Sinkhorn algorithm to accelerate optimization.
Consequently, in the following we will also refer to the entropy regularization parameter as $\veps$ and to the gradual reduction scheme as $\veps$-scaling.
The ideas of \cite{YuilleInvisibleHand1994} are refined in \cite{SharifySinkhorn2013}. In particular, the latter proves convergence of a modified algorithm with `deformed iterations' where $\veps$ is gradually decreased during the iterations, similar to $\veps$-scaling.
They show that the primal iterate converges to the unregularized solution if the decrease is sufficiently slow. Unfortunately, the number of iterations to reach a given value of $\veps$ increases exponentially, as $\veps$ decreases. Thus it is ``mostly interesting from the theoretical point of view'' \cite[p.~8]{SharifySinkhorn2013}.

\paragraph{Limitations of Entropic Transport}
Despite its considerable merits, there are some fundamental constraints to the naive entropy regularization approach.
Entropy introduces some blur in the optimal assignment. While this may sometimes be beneficial (see above), in many applications it is considered a nuisance (e.g.~it quickly smears distinct features in gradient flows), and one would like to run the scaling algorithm with as little regularization as possible.
However, a standard implementation has some major numerical limitations, becoming increasingly severe as the regularization approaches zero.
The diagonal scaling factors diverge in the limit of vanishing regularization, leading to numerical overflow and instabilities.
Moreover, the algorithm requires an increasing number of iterations to converge.
In practice this can often be remedied by $\veps$-scaling, but its efficiency is not yet well understood theoretically.
Therefore, numerically this limit is difficult to reach.
In addition, naively storing the dense kernel matrix requires just as much memory as storing the full cost matrix in standard linear programming solvers and multiplications with the kernel matrix become increasingly slow.
Thus, effective heuristics to avoid storing of, and multiplication by, the dense kernel matrix have been conceived, such as efficient Gaussian convolutions or approximation by a pre-factored heat kernel \cite{Solomon-siggraph-2015}.
However, these remedies only work for particular (although relevant) problems, and do not solve the issues of blur and diverging scaling factors.

\subsection{Contribution and Outline}
\label{sec:IntroductionContribution}
In Section \ref{sec:Background} we recall the framework for transport-type problems and corresponding scaling algorithms for their entropy regularized counterparts, as put forward in \cite{ChizatEntropicNumeric2018}.
The main contributions of this article are twofold:
In Section \ref{sec:Algorithm} we propose to combine four modifications of the Sinkhorn algorithm to address issues with numerical instability, slow convergence and large kernel matrices.
In Section \ref{sec:Auction} a new convergence analysis for the Sinkhorn algorithm is derived, based on an analogy to the auction algorithm.
The two sections can be read independently from each other.
The modifications used in Section \ref{sec:Algorithm} are:
\begin{itemize}
	\item \emph{Section \ref{sec:AlgorithmLogDomain}:} A log-domain stabilization of the Sinkhorn algorithm, as described in \cite{ChizatEntropicNumeric2018}. It allows to numerically run the algorithm at small regularizations while largely retaining the simple matrix scaling structure.
	\item \emph{Section \ref{sec:AlgorithmEpsScaling}:} The well-known $\veps$-scaling heuristic, to reduce the number of required iterations.
	\item \emph{Section \ref{sec:AlgorithmSparseKernel}:} Sparsification of the kernel matrix by adaptive truncation, to reduce memory demand and accelerate iterations. We quantify the error induced by truncation and propose a truncation scheme which reliably yields small error bounds that are easy to evaluate.
		While truncation has been proposed elsewhere (e.g.~\cite{MandadSurfaceMatching2017}), to the best of our knowledge the present article gives the first concrete bounds for the inflicted error.
	\item \emph{Section \ref{sec:AlgorithmMultiScale}:} A multi-scale scheme, inspired, for instance, by \cite{SchmitzerSchnoerr-SSVM2013,SchmitzerShortCuts2015,ObermanOptimalTransportationLP2015}.
		This serves two purposes:
		First, it allows for a more efficient computation of the truncated kernel.
		Second, we propose to combine the coarse-to-fine approach with simultaneous $\veps$-scaling, which drastically reduces the number of variables during early stages of $\veps$-scaling, without losing significant precision.
\end{itemize}
We emphasize that each modification builds on the previous ones (see Remark \ref{rem:EnhancementRelation}) and only combining all four leads to an algorithm that can solve large problems with significantly less runtime, memory and regularization, as compared to the naive algorithm.
The adaptations extend to the more general scaling algorithms for transport-type problems presented in \cite{ChizatEntropicNumeric2018}.

In Section \ref{sec:Auction} we develop a new convergence analysis of the Sinkhorn algorithm, based on analogy to the auction algorithm, different from the Hilbert metric approach of \cite{FranklinLorenz-Scaling-1989}.
The structure of Section \ref{sec:Auction} is:
\begin{itemize}
	\item \emph{Section \ref{sec:AuctionRecall}:} The classical auction algorithm for the linear assignment problem is recalled.
	\item \emph{Section \ref{sec:AuctionSimple}:} A slightly modified asymmetric variant of the Sinkhorn algorithm is given and a bound is derived for the number of iterations until a prescribed accuracy is reached. As for the auction algorithm, for fixed $\veps$ the maximal number of iterations scales as $\mc{O}(1/\veps)$. This is in good agreement with numerical experiments (cf.~Section \ref{sec:NumericsEfficiency}). To avoid the difficulties with slow convergence in Hilbert's projective metric (cf.~Section \ref{sec:IntroductionRelatedWork}) we choose a weaker, but reasonable, measure of convergence (cf.~Remark \ref{rem:SinkhornAsymmetricStoppingCriterion}).
	\item \emph{Section \ref{sec:AuctionStability}:} We prove stability of optimal dual solutions of entropy regularized OT under changes of the regularization parameter. 
This also implies stability of dual solutions in the limit of vanishing regularization and therefore complements results of \cite{Cominetti-ExpBarrierConvergence-1992} (see also Remark \ref{rem:AuctionStabilityMotivation}).
	\item \emph{Section \ref{sec:AuctionEpsScaling}:} Our eventual goal is a better theoretical understanding of the $\veps$-scaling heuristic and its efficiency. We show that the above stability result is an important step and discuss missing steps for a full proof.
	To our knowledge (with the exception of \cite{SharifySinkhorn2013}, see above), these are the first theoretical results towards $\veps$-scaling for the Sinkhorn algorithm.
\end{itemize}
Numerical experiments confirm the efficiency of the modified algorithm (Section \ref{sec:NumericsEfficiency}). Examples for unbalanced optimal transport, barycenters, and Wasserstein gradient flows illustrate that the modified algorithms retain the versatility of the diagonal scaling algorithms presented in \cite{BenamouIterativeBregman2015,PeyreEntropicFlows2015,ChizatEntropicNumeric2018} (Section \ref{sec:NumericsVersatility}).

\subsection{Notation and Preliminaries}
\label{sec:IntroNotation}
We assume that the reader has a basic knowledge of convex optimization, such as convex conjugation, Fenchel--Rockafellar duality and primal-dual gaps (cf.~\cite{ConvexFunctionalAnalysis-11}).

Throughout this article, we will consider transport problems between two discrete finite spaces $X$ and $Y$.
For a discrete, finite space $Z$ (typically $X$, $Y$ or $X \times Y$) we identify functions and measures over $Z$ with vectors in $\R^{|Z|}$, which we simply denote by $\R^Z$. For $v \in \R^{Z}$, $z \in Z$ we write $v(z)$ for the component of $v$ corresponding to $z$ (subscript notation is reserved for other purposes). The standard Euclidean inner product is denoted by $\la \cdot, \cdot \ra$.
The sets of vectors with positive and strictly positive entries are denoted by $\R_+^{Z}$ and $\R_{++}^{Z}$. The probability simplex over $Z$ is denoted by $\prob(Z)$.
We write $\RExt \eqdef \R \cup \{-\infty,+\infty\}$ for the extended real line and $\RExt^Z$ for the space of vectors with possibly infinite components.

For $a$, $b \in \R^Z$ the operators $\odot$ and $\oslash$ denote pointwise multiplication and division, e.g.~$a \odot b \in \R^Z$, $(a \odot b)(z) \eqdef a(z) \cdot b(z)$ for $z \in Z$. The functions $\exp$ and $\log$ are extended to $\R^Z$ by pointwise application to all components: $\exp(a)(z) \eqdef \exp(a(z))$.
We write $a \geq b$ if $a(z) \geq b(z)$ for all $z \in Z$, $a \geq 0$ if $a(z) \geq 0$ for all $z \in Z$  (and likewise for $\leq$, $>$ and $<$).
For $a \in \R$, $a_Z$ denotes the vector in $\R^Z$ with all entries being $a$.
We write $\max a$ and $\min a$ for the maximal and minimal entry of $a$.

For $\mu \in \R^Z$ and a subset $A \subset Z$ we also use the notation $\mu(A) \eqdef \sum_{z \in A} \mu(z)$, analogous to measures.
We say $\mu \in \R^Z$ is absolutely continuous w.r.t.~$\nu \in \R_+^Z$ and write $\mu \ll \nu$ when $[\nu(z) = 0] \Rightarrow [\mu(z) = 0]$. This is the discrete special case of absolute continuity for measures. The set $\spt \mu \eqdef \{z \in Z : \mu(z) \neq 0\}$ is called support of $\mu$.
The power set of $Z$ is denoted by $2^Z$.

For a subset $\mc{C} \subset \R^Z$ the indicator function of $\mc{C}$ over $\R^Z$ is given by $\iota_{\mc{C}}(v)=0$ if $v \in \mc{C}$ and $+\infty$ else.
In particular, for $v$, $w \in \R^Z$ one finds $\iota_{\{v\}}(w) = 0$ if $v=w$ and $+\infty$ otherwise.
Moreover, we merely write $\iota_+$ for $\iota_{\R^Z_+}$.
For $v \in \R^Z$ we introduce the short notation $\iota_{\leq v} : \R^Z \to \RExt$ with $\iota_{\leq v}(w)=0$ if $w(z) \leq v(z)$ for all $z \in Z$ and $+\infty$ otherwise.

The projection matrices $\projx \in \R^{X \times (X \times Y)}$ and $\projy \in \R^{Y \times (X \times Y)}$ are given by
\begin{align*}
	\projx(x,(x',y')) & \eqdef \begin{cases}
		1 & \tn{if } x = x', \\
		0 & \tn{else.}
	\end{cases}
	&
	\projy(y,(x',y')) & \eqdef \begin{cases}
		1 & \tn{if } y = y', \\
		0 & \tn{else.}
	\end{cases}
	\intertext{They act on some $\pi \in \R^{X \times Y}$ as follows:}
	(\projx\,\pi)(x) & = \sum_{y \in Y} \pi(x,y) = \pi(\{x\} \times Y),
	&
	(\projy\,\pi)(y) & = \sum_{x \in X} \pi(x,y) = \pi(X \times \{y\}).
\end{align*}
	That is, they give the $X$ and $Y$ marginal in the sense of measures. Conversely, for some $v \in \R^X$, $w \in \R^Y$ we find $(\projxT v)(x,y) = v(x)$ and $(\projyT w)(x,y) = w(y)$.

\begin{definition}[Kullback--Leibler Divergence]
	\label{def:KullbackLeibler}
	For $\mu$, $\nu \in \R^Z$ the Kullback--Leibler divergence of $\mu$ w.r.t.~$\nu$ is given by
	\begin{align}
		\KL(\mu|\nu) & \eqdef \begin{cases}
			\sum_{\substack{z \in Z:\\\mu(z)>0}} \mu(z) \log\left(\tfrac{\mu(z)}{\nu(z)}\right)
				- \mu(Z) + \nu(Z) & \tn{if } \mu, \nu \geq 0, \mu \ll \nu\,, \\
			+ \infty & \tn{else.}
			\end{cases}
	\end{align}
	The convex conjugate w.r.t.~the first argument is given by $\KL^\ast(\alpha|\nu)  = \sum_{z \in Z} \left(\exp(\alpha(z))-1\right) \cdot \nu(z)$.
	The $\KL$ divergence plays a central role in this article and is used on various different base spaces. Sometimes, when referring to the $\KL$ divergence on a space $Z$, we will add a subscript $\KL_Z$ for clarification.
\end{definition}

\begin{definition}[KL Proximal Step]
	\label{def:KLProx}
	For a convex, lower semicontinuous function $f : \R^Z \to \RExt$ and a step size $\tau>0$ the proximal step operator for the Kullback--Leibler divergence is given by
	\begin{align}
		\ProxKL{f}{1/\tau} & : \R^Z \to \R^Z, &
		\mu \mapsto \argmin_{\nu \in \R^Z} \left( \tfrac{1}{\tau} \KL(\nu|\mu) + f(\nu) \right)\,.
	\end{align}
	A unique minimizer exists, if there is some $\nu \in \R^Z$, $\nu \ll \mu$ such that $f(\nu) \neq \pm \infty$.
	Throughout this article we shall always assume that this is the case.
\end{definition}
For Sect.~\ref{sec:Auction} we require the following Lemma.
\begin{lemma}[Softmax and Softmin]
	\label{lem:SoftMaxMin}
	For a parameter $\veps > 0$ and $a \in \R^Z$ let
	\begin{align*}
		\softmax(a,\veps) & \eqdef \veps\,\log\left(\sum_{z \in Z} \exp(a(z)/\veps) \right), &
		\softmin(a,\veps) & \eqdef -\veps\,\log\left(\sum_{z \in Z} \exp(-a(z)/\veps) \right).
	\end{align*}
	For $\veps$, $\lambda > 0$ and $a$, $b \in \R^Z$ one has the relations
	\begin{subequations}
	\label{eqn:Softbounds}
	\begin{gather}
		\label{eqn:SoftboundsMax}
		\max(a) \leq \softmax(a,\veps) \leq \max(a) + \veps \, \log |Z| , \\
		\label{eqn:SoftboundsMin}
		\min(a) - \veps\,\log |Z| \leq \softmin(a,\veps) \leq \min(a) , \\
		\min(a-b) - \lambda \, \log |Z| \leq \softmax(a,\veps) - \softmax(b,\lambda)
			\leq \max(a-b) + \veps \, \log |Z|, \\
		\min(a-b) - \veps \, \log |Z| \leq \softmin(a,\veps) - \softmin(b,\lambda)
			\leq \max(a-b) + \lambda \, \log |Z|.
	\end{gather}
	\end{subequations}
\end{lemma}
\begin{proof}
	The first line follows immediately from $0 \leq \exp(a(z)/\veps) \leq \exp(\max a/\veps)$.
	Line three then follows from $\min(a-b) \leq \max(a)-\max(b) \leq \max(a-b)$.
	The second and fourth line are implied by $\softmin(a,\veps)=-\softmax(-a,\veps)$.
\end{proof}


\section{Entropy Regularized Transport-Type Problems and Diagonal Scaling Algorithms}
\label{sec:Background}

\subsection{Transport-Type Problems}
\label{sec:BackgroundTransport}
For two probability measures $\mu \in \prob(X)$ and $\nu \in \prob(Y)$ the set $\Pi(\mu,\nu) \eqdef \{\pi \in \prob(X \times Y) \,\colon\, \projx\,\pi = \mu, \projy\,\pi=\nu\}$ is called the couplings or transport plans between $\mu$ and $\nu$. A coupling $\pi$ describes a rearrangement of the mass of $\mu$ into $\nu$, $\pi(x,y)$ can be interpreted as the mass taken from $x$ to $y$.
Let $c \in \RExt^{X \times Y}$ be a cost function, such that the cost of taking one unit of mass from $x \in X$ to $y \in Y$ is given by $c(x,y)$. The cost inflicted by a coupling $\pi$ is then given by $\la c, \pi \ra$ and the optimal transport problem between $\mu$ and $\nu$ is given by $\min \{ \la c,\pi \ra | \pi \in \Pi(\mu,\nu) \}$.
This means, we are looking for the most cost-efficient mass rearrangement between $\mu$ and $\nu$.
Note that for $\pi \in \R^{X \times Y}$ one can write $\iota_{\Pi(\mu,\nu)}(\pi) = \iota_{\{\mu\}}(\projx\,\pi) + \iota_{\{\nu\}}(\projy\,\pi) + \iota_+(\pi)$ where the first two terms represent the marginal constraints and the last term ensures that $\pi$ is non-negative. Then we can reformulate the problem as
\begin{align}
	\label{eqn:PrototypeOT}
	\min_{\pi \in \R^{X \times Y}} \iota_{\{\mu\}}(\projx\,\pi) + \iota_{\{\nu\}}(\projy\,\pi)
		+ \la c, \pi \ra + \iota_+(\pi)\,.
\end{align}

Recently it has been proposed to replace the constraints $\projx \pi=\mu$ and $\projy \pi = \nu$ by soft constraints. This allows meaningful comparison between measures of different total mass. Such formulations were studied e.g.~in \cite{LieroMielkeSavare-HellingerKantorovich-2015a} (see also \cite{ChizatEntropicNumeric2018} for more context). A particularly relevant choice for the soft constraints is the Kullback--Leibler divergence.
A corresponding `unbalanced' transport problem is given by
\begin{align}
	\label{eqn:PrototypeUnbalanced}
	\min_{\pi \in \R^{X \times Y}} \lambda \cdot \KL(\projx\,\pi|\mu)
		+ \lambda \cdot \KL(\projy\,\pi|\nu)
		+ \la c,\pi \ra
		+ \iota_+(\pi)\,.
\end{align}
where $\lambda>0$ is a weighting parameter.
Note that neither $\mu$, $\nu$ nor $\pi$ need to be probability measures in this case and each may have different total mass.

When $X=Y$ is a metric space with metric $d$, for $\lambda=1$ and the cost function $c=d^2$, the square root of the optimal value of \eqref{eqn:PrototypeUnbalanced} yields the so called Gaussian Hellinger--Kantorovich (GHK) distance on $\R_+^X$, introduced in \cite{LieroMielkeSavare-HellingerKantorovich-2015a}.
	Similarly, for the cost function
	\begin{align}
		\label{eqn:WFRCost}
		c(x,y) & \eqdef \begin{cases}
			-\log\left([\cos(d(x,y))]^2\right) & \tn{if } d(x,y) < \pi/2 \\
			+\infty & \tn{else.}
			\end{cases}
	\end{align}
	one obtains the Wasserstein--Fisher--Rao (WFR) distance (or Hellinger--Kantorovich distance), introduced independently and simultaneously in \cite{KMV-OTFisherRao-2015,ChizatOTFR2015,LieroMielkeSavare-HellingerKantorovich-2015a}.
WFR is the length distance induced by GHK \cite{LieroMielkeSavare-HellingerKantorovich-2015a}.

Problems \eqref{eqn:PrototypeOT} and \eqref{eqn:PrototypeUnbalanced} share a common structure: in both we optimize over non-negative measures $\pi$ on the product space $X \times Y$, there is a linear cost term $\la c,\pi \ra$ and two functions act on the marginals of $\pi$.
They are prototypical examples of a family of transport-type optimization problems with a common functional structure that was introduced in \cite{ChizatEntropicNumeric2018}.
The general structure is given in the following definition.
\begin{definition}[Generic Transport-Type Problem]
	\label{def:GenericFormulation}
	For two convex marginal functions $F_X : \R^X \to \RExt$, $F_Y : \R^Y \to \RExt$ and a cost function $c \in \RExt^{X \times Y}$ the primal transport-type problem is given by:
	\begin{subequations}
	\label{eqn:Generic}
	\begin{align}
		\label{eqn:GenericPrimal}
		\min_{\pi \in \R^{X \times Y}} E(\pi) & \qquad \tn{with} &
		E(\pi) & \eqdef F_X(\projx\,\pi) + F_Y(\projy\,\pi)
			+ \la c,\pi \ra + \iota_+(\pi) \\
	\intertext{The corresponding dual problem is given by:}
		\label{eqn:GenericDual}
		\max_{(\alpha,\beta) \in (\R^X,\R^Y)} J(\alpha,\beta) & \qquad \tn{with} &
		J(\alpha,\beta) & \eqdef -F_X^\ast(-\alpha) - F_Y^\ast(-\beta) - \iota_{\leq c}(\projxT\,\alpha + \projyT\,\beta)
	\end{align}
	\end{subequations}
	The indicator function $\iota_{\leq c}(\projxT\,\alpha + \projyT\,\beta)$ denotes the classical optimal transport dual constraint $\alpha(x) + \beta(y) \leq c(x,y)$ for all $(x,y) \in X \times Y$ (see Section \ref{sec:IntroNotation}).
\end{definition}

This family also covers Wasserstein gradient flows and the structure can be extended to multiple couplings to describe barycenter and multi-marginal problems (see \cite{BenamouIterativeBregman2015,ChizatEntropicNumeric2018} for details).
As indicated, the standard optimal transport problem \eqref{eqn:PrototypeOT} is obtained as a special case.
\begin{definition}[Standard Optimal Transport]
	\label{def:OptimalTransport}
	Problem \eqref{eqn:PrototypeOT} is a special case of Def.~\ref{def:GenericFormulation} with $F_X \eqdef \iota_{\{\mu\}}$ and $F_Y \eqdef \iota_{\{\nu\}}$. The primal and dual functional are given by:
	\begin{subequations}
	\label{eqn:OptimalTransport}
	\begin{align}
		E(\pi) & = \iota_{\{\mu\}}(\projx\,\pi) + \iota_{\{\nu\}}(\projy\,\pi)
			+ \la c,\pi \ra + \iota_+(\pi) \\
		\label{eqn:OptimalTransportDual}
		J(\alpha,\beta) & = \la \alpha,\mu \ra + \la \beta,\nu \ra - \iota_{\leq c}(\projxT\,\alpha + \projyT\,\beta)
	\end{align}
	\end{subequations}
\end{definition}

Likewise, we can proceed for the unbalanced transport problem \eqref{eqn:PrototypeUnbalanced}.
\begin{definition}[Unbalanced Optimal Transport with KL Fidelity]
	\label{def:UnbalancedTransport}
	Problem \eqref{eqn:PrototypeUnbalanced} is a special case of Def.~\ref{def:GenericFormulation} with $F_X \eqdef \lambda \cdot \KL(\cdot|\mu)$ and $F_Y \eqdef \lambda \cdot \KL(\cdot|\nu)$. The primal and dual functional are given by:
	\begin{align}
		\label{eqn:UnbalancedTransportPrimal}
		E(\pi) & = \lambda \cdot \KL(\projx\,\pi|\mu) + \lambda \cdot \KL(\projy\,\pi|\nu)
			+ \la c,\pi \ra + \iota_+(\pi) \\
		\label{eqn:UnbalancedTransportDual}
		J(\alpha,\beta) & = -\lambda \cdot \KL^\ast(-\alpha/\lambda) - \lambda \cdot \KL^\ast(-\beta/\lambda) - \iota_{\leq c}(\projxT\,\alpha + \projyT\,\beta)
	\end{align}
\end{definition}

\subsection{Entropy Regularization and Diagonal Scaling Algorithms}
\label{sec:BackgroundEntropy}

Now we apply entropy regularization to the above transport-type problems (see Sect.~\ref{sec:IntroductionRelatedWork} for references) and replace the non-negativity constraint in \eqref{eqn:GenericPrimal} by the Kullback--Leibler divergence. For this we need to select some reference measure $\rho \in \R_+^{X \times Y}$. We then replace the term $\iota_+(\pi)$ in \eqref{eqn:GenericPrimal} by $\veps \cdot \KL(\pi|\rho)$, where $\veps>0$ is a regularization parameter. Then one typically `pulls' the linear cost term into the $\KL$ divergence:
\begin{gather}
	\la c, \pi \ra + \veps\, \KL(\pi|\rho) = \veps\, \KL(\pi|\kernel) + \veps \cdot \big(\rho(X \times Y) - \kernel(X \times Y) \big) \nonumber \\
	\label{eqn:Kernel}
	\tn{where $\kernel \in \R_+^{X \times Y}$ with} \quad \kernel(x,y) \eqdef \exp(-c(x,y)/\veps) \cdot \rho(x,y)\,.
\end{gather}
with the convention $\exp(-\infty)=0$.
$\kernel$ is called the kernel associated with $c$ and the regularization parameter $\veps$.
For convenience we formally introduce the function
\begin{align}
	\label{eqn:getKernel}
	\getKernel & : \R_{++} \to \R^{X \times Y}, &
	\veps & \mapsto \exp(-c/\veps) \odot \rho\,.
\end{align}
We obtain the regularized equivalent to Def.~\ref{def:GenericFormulation}.

\begin{definition}[Regularized Generic Formulation]
	\label{def:RegularizedGenericFormulation}
	\begin{subequations}
	\label{eqn:RegularizedGeneric}
	\begin{align}
		\label{eqn:RegularizedGenericPrimal}
		\min_{\pi \in \R_+^{X \times Y}} E(\pi) & \qquad \tn{with} &
		E(\pi) & \eqdef F_X(\projx \pi) + F_Y(\projy \pi) + \veps\,\KL(\pi|\kernel) \\
		\label{eqn:RegularizedGenericDual}
		\max_{(\alpha,\beta) \in (\R^{X},\R^Y)} J(\alpha,\beta) & \qquad \tn{with} &
		J(\alpha,\beta) & \eqdef -F_X^\ast(-\alpha) - F_Y^\ast(-\beta) - \veps\,\KL^\ast\left(
			[\projxT \alpha + \projyT \beta]/\veps \big| \kernel \right)
	\end{align}
	\end{subequations}
	Primal optimizers $\pi^\dagger$ have the form 
	\begin{align}
		\label{eqn:EntropyPDRelation}
		\pi^\dagger = \diag(\exp(\alpha^\dagger/\veps))\,\kernel\,\diag(\exp(\beta^\dagger/\veps))
	\end{align}
	where $(\alpha^\dagger,\beta^\dagger)$ are dual optimizers. Conversely, for dual optimizers $(\alpha^\dagger,\beta^\dagger)$, $\pi^\dagger$ constructed as above is primal optimal \cite{ChizatEntropicNumeric2018}.
\end{definition}
Intuitively we see the relation between \eqref{eqn:Generic} and \eqref{eqn:RegularizedGeneric} as $\veps \to 0$. For example, the term $\veps\,\KL^\ast\big([\projxT \alpha + \projyT \beta]/\veps \big| \kernel \big)$ in \eqref{eqn:RegularizedGenericDual} can be interpreted as a smooth barrier function for the dual constraint $\projxT \alpha + \projyT \beta \leq c$ in \eqref{eqn:GenericDual}. We refer to Sect.~\ref{sec:IntroductionRelatedWork} for references to rigorous convergence results.

Under suitable assumptions problem \eqref{eqn:RegularizedGenericDual} can be solved by alternating optimization in $\alpha$ and $\beta$ (see \cite{ChizatEntropicNumeric2018} for details).
For fixed $\beta$, consider the $\KL^\ast$-term:
\begin{align*}
	\KL_{X \times Y}^\ast\left(
			[\projxT \alpha + \projyT \beta]/\veps \big| \kernel \right)
	= \KL_X^\ast\left( \alpha/\veps | \kernel\,\exp(\beta/\veps) \right) + \sum_{(x,y) \in X \times Y}
		\kernel(x,y)\,\left(\exp(\beta(y)/\veps)-1\right).
\end{align*}
Note that the last term is constant w.r.t.\ $\alpha$. Therefore, optimizing \eqref{eqn:RegularizedGenericDual} over $\alpha$, for fixed $\beta$ corresponds to maximizing
\begin{align}
	\label{eqn:PartialDual}
	J_X(\alpha) & = -F_X^\ast(-\alpha) - \veps\,\KL_X^\ast\left(\alpha/\veps | \kernel\,\exp(\beta/\veps) \right),
\\
\intertext{where $\kernel\,\exp(\beta/\veps)$ denotes standard matrix vector multiplication. The corresponding primal problem consists of minimizing}
	\label{eqn:PartialPrimal}
	E_X(\sigma) & = F_X(\sigma) + \veps\,\KL_X(\sigma| \kernel\,\exp(\beta/\veps))\,.
\end{align}
This is a proximal step of $F_X$ for the $\KL$ divergence with step size $1/\veps$ (see Def.~\ref{def:KLProx}). So, by using the PD-optimality conditions between \eqref{eqn:PartialDual} and \eqref{eqn:PartialPrimal} (see e.g.~\cite[Thm.~19.1]{ConvexFunctionalAnalysis-11}), for a given $\beta$ the primal optimizer $\sigma^\dagger$ of \eqref{eqn:PartialPrimal} and the dual optimizer $\alpha^\dagger$ of \eqref{eqn:PartialDual} are given by
\begin{align}
	\label{eqn:PartialOptimisation}
	\sigma^\dagger & = \ProxKL{F_X}{\veps}(\kernel\,\exp(\beta/\veps)), &
	\alpha^\dagger & = \veps \, \log(\sigma^\dagger \oslash (\kernel\,\exp(\beta/\veps))),
\end{align}
Analogously, optimization w.r.t.\ $\beta$ for fixed $\alpha$ is related to $\KL$ proximal steps of $F_Y$.
Starting from some initial $\iterz{\beta}$, we can iterate alternating optimization to obtain a sequence $\iterz{\beta},\iter{\alpha}{1}, \iter{\beta}{1}, \iter{\alpha}{2}, \ldots$ as follows:
\begin{subequations}
\label{eqn:IterationsDual}
\begin{align}
	\iterll{\alpha} & \eqdef \veps\,\log\left( \ProxKL{F_X}{\veps}(\kernel\,\exp(\iterl{\beta}/\veps))
		\oslash [\kernel\,\exp(\iterl{\beta}/\veps)] \right), \\
	\iterll{\beta} & \eqdef \veps\,\log\left( \ProxKL{F_Y}{\veps}(\kernel^\T \exp(\iterll{\alpha}/\veps))
		\oslash [\kernel^\T\exp(\iterll{\alpha}/\veps)] \right).
\end{align}
\end{subequations}
The algorithm becomes somewhat simpler when it is formulated in terms of the effective variables
\begin{align}
	\label{eqn:ScalingFactors}
	u & \eqdef \exp(\alpha/\veps)\,, &
	v & \eqdef \exp(\beta/\veps)\,.
\end{align}
For more convenient notation we introduce the $\proxdivSymb$ operator of a function $F$ and step size $1/\veps$:
\begin{align}
	\label{eqn:Proxdiv}
	\proxdiv{F}{\veps} : \sigma \mapsto \ProxKL{F}{\veps}(\sigma) \oslash \sigma
\end{align}
The iterations then become:
\begin{align}
	\label{eqn:IterationsScaling}
	\iterll{u} & \eqdef \proxdiv{F_X}{\veps}(\kernel\,\iterl{v})\,, &
	\iterll{v} & \eqdef \proxdiv{F_Y}{\veps}(\kernel^\T\iterll{u})\,.
\end{align}
The primal-dual relation \eqref{eqn:EntropyPDRelation} then becomes $\pi^\dagger = \diag(u^\dagger)\,\kernel\,\diag(v^\dagger)$, which is why $u$ and $v$ are often referred to as diagonal scaling factors.

\begin{remark}
	\label{rem:DualVariablesScaling}
	Throughout this article, we will refer to the arguments of the dual functionals \eqref{eqn:GenericDual} and \eqref{eqn:RegularizedGenericDual} as \emph{dual variables} and denote them with $(\alpha,\beta)$.
	The effective, exponentiated variables, introduced in \eqref{eqn:ScalingFactors}, will be denoted by $(u,v)$ and referred to as \emph{scaling factors}.
\end{remark}

For future reference let us state the full scaling algorithm.
\begin{algorithmthm}[Scaling Algorithm]\hfill
\label{alg:Scaling}
\begin{algorithmic}[1]
\Function{ScalingAlgorithm}{$\veps$,$\iterz{v}$}
	\State $K \gets \getKernel(\veps)$\SemiCol
		$v \gets \iterz{v}$
		\Comment{compute kernel, see \eqref{eqn:getKernel}; initialize scaling variable}
	\Repeat
		\State $u \gets \proxdiv{F_X}{\veps}(\kernel\,v)$\SemiCol
			$v \gets \proxdiv{F_Y}{\veps}(\kernel^\T u)$
	\Until{stopping criterion}
	\State \Return $(u,v)$
\EndFunction
\end{algorithmic}
\end{algorithmthm}
The stopping criterion is typically a bound on the primal-dual gap between dual iterates $(\alpha,\beta)=\veps\,\log(u,v)$ and primal iterate $\pi=\diag(u)\,\kernel\,\diag(v)$, an error bound on the marginals of $\pi$ (for standard optimal transport) or a pre-determined number of iterations.

With alternating iterations \eqref{eqn:IterationsDual} or \eqref{eqn:IterationsScaling} a large family of functionals of form \eqref{eqn:RegularizedGenericPrimal} can be optimized, as long as the $\KL$ proximal steps of $F_X$ and $F_Y$ can be computed efficiently. A particularly relevant sub-family is, where $F_X$ and $F_Y$ are separable and are a sum of pointwise functions. Then the $\KL$ steps decompose into pointwise one-dimensional $\KL$ steps, see \cite[Section 3.4]{ChizatEntropicNumeric2018} for details.

Since Section \ref{sec:Auction} focusses on the special case of entropy regularized optimal transport, let us explicitly state the corresponding functional and iterations.

\begin{definition}[Entropic Optimal Transport]
	\label{def:EntropicOptimalTransport}
	For marginals $\mu \in \prob(X)$, $\nu \in \prob(Y)$ and a cost function $c \in \RExt^{X \times Y}$ the entropy regularized optimal transport problem is obtained from Def.~\ref{def:RegularizedGenericFormulation} by setting $F_X \eqdef \iota_{\{\mu\}}$, $F_Y \eqdef \iota_{\{\nu\}}$ (see Definition \ref{def:OptimalTransport} for the unregularized functional). We find:
	\begin{subequations}
	\begin{align}
		\label{eqn:EntropicOTPrimal}
		E(\pi) & = \iota_{\{\mu\}}(\projx\,\pi) + \iota_{\{\nu\}}(\projy\,\pi) + \veps\,\KL(\pi|\kernel) \\
		\label{eqn:EntropicOTDual}
		J(\alpha,\beta) & = \la \alpha,\mu \ra + \la \beta,\nu \ra - \veps\,\KL^\ast\left(
			[\projxT \alpha + \projyT \beta]/\veps \big| \kernel \right)		
	\end{align}
	\end{subequations}
	The proximal steps of $F_X$ and $F_Y$ are trivial (if $\kernel$ has non-empty columns and rows) and we recover the famous Sinkhorn iterations:
	\begin{subequations}
	\label{eqn:ProxdivSinkhorn}	
	\begin{align}
		\proxdiv{F_X}{\veps}(\sigma) & = \mu \oslash \sigma\,, &
		\proxdiv{F_Y}{\veps}(\sigma) & = \nu \oslash \sigma\,, \\
		\iterll{u} & = \mu \oslash (\kernel\,\iterl{v})\,, &
		\iterll{v} & = \nu \oslash (\kernel^\T \iterll{u})\,.
	\end{align}
	\end{subequations}
\end{definition}


\newcommand{\alphaBase}{\hat{\alpha}}
\newcommand{\betaBase}{\hat{\beta}}
\newcommand{\uRel}{\tilde{u}}
\newcommand{\vRel}{\tilde{v}}

\newcommand{\kernelStable}{\mc{K}}
\newcommand{\getKernelStable}{\tn{get}\mc{K}}
\newcommand{\getKernelSparse}{\tn{get}\hat{K}}

\newcommand{\kernelBase}{\mc{K}}
\newcommand{\kernelSparse}{\hat{K}}
\newcommand{\costSparse}{\hat{c}}

\newcommand{\getNeigh}{\tn{get}\mc{N}}

\newcommand{\children}{\tn{children}}
\newcommand{\parent}{\tn{parent}}

\section{Stabilized Sparse Multi-Scale Algorithm}
\label{sec:Algorithm}
Throughout this section we combine four adaptions to the Algorithm \ref{alg:Scaling} to overcome the limitations of a naive implementation outlined in Section \ref{sec:IntroductionRelatedWork}.

\subsection{Log-Domain Stabilization}
\label{sec:AlgorithmLogDomain}
When running Algorithm \ref{alg:Scaling} with small regularization parameter $\veps$, entries in the kernel $\kernel$, and the scaling factors $u$ and $v$ may become both very small and very large, leading to numerical difficulties.
However, under suitable conditions (e.g.~standard optimal transport, finite cost function) it can be shown that the optimal dual variables $(\alpha,\beta)$ remain finite and have a stable limit as $\veps \to 0$ (\cite{Cominetti-ExpBarrierConvergence-1992}, see also Remark \ref{rem:AuctionStabilityMotivation}).
In \cite{YuilleInvisibleHand1994,SharifySinkhorn2013} and others it was proposed to formulate the Sinkhorn iterations directly in terms of the dual variables, instead of the scaling factors. For example, an update of $\alpha$ would be performed as follows:
\begin{subequations}
\label{eqn:SinkhornLogDomain}
\begin{align}
	\iterll{\psi}(x,y) & = -c(x,y)+\iterl{\beta}(y), \quad
	\iterll{\tilde{\psi}}(x,y) = \iterll{\psi}(x,y) - \max_{y' \in Y} \iterll{\psi}(x,y') \\
	\iterll{\alpha}(x) & = \veps\,\log\mu(x) - \veps \log\Big( \sum_{y \in Y} \exp(\iterll{\tilde{\psi}}(x,y)/\veps) \cdot \rho(x,y) \Big) - \max_{y \in Y} \iterll{\psi}(x,y)
\end{align}
\end{subequations}
Subtracting the maximum from $\iterll{\psi}$ avoids large arguments in the exponential function.
While this resolves the issue of extreme scaling factors, it perturbs the simple matrix multiplication structure of the algorithm and requires many additional evaluations of $\exp$ and $\log$ in each iteration.

As an alternative, we employ the redundant parametrization of the iterations as proposed in\cite{ChizatEntropicNumeric2018}. The scaling factors $(u,v)$, \eqref{eqn:ScalingFactors}, are written as
\begin{align}
	\label{eqn:RelativeScalingFactors}
	u & = \uRel \odot \exp(\alphaBase/\veps)\,, &
	v & = \vRel \odot \exp(\betaBase/\veps)\,.
\end{align}
Our goal is to formulate iterations \eqref{eqn:IterationsScaling} directly in terms of $(\uRel,\vRel)$, while keeping $(\alphaBase,\betaBase)$ unchanged during most iterations. The role of $(\alphaBase,\betaBase)$ is to occasionally `absorb' the large values of $(u,v)$ such that $(\uRel,\vRel)$ remain bounded.
This leads to two types of iterations: stabilized iterations, during which only $(\uRel,\vRel)$ are changed, and absorption iterations, during which $(\uRel,\vRel)$ are absorbed into $(\alphaBase,\betaBase)$.
In this way, we can combine the simplicity of the scaling algorithm in terms of the scaling factor formulation with the numerical stability of the iterations in the log-domain formulation \eqref{eqn:SinkhornLogDomain}.

Analogous to the function $\getKernel$, \eqref{eqn:getKernel}, we define the \emph{stabilized kernel} as
\begin{subequations}
\label{eqn:StabilizedKernel}
\begin{align}
	\label{eqn:StabilizedKernelDiag}
	\getKernelStable & : \R^X \times \R^Y \times \R_{++} \to \R^{X \times Y}, &
	(\alpha,\beta,\veps) & \mapsto \diag(\exp(\alpha/\veps))\,\getKernel(\veps)\,\diag(\exp(\beta/\veps)),
\end{align}
\begin{align}
	\label{eqn:StabilizedKernelEntries}
	[\getKernelStable(\alpha,\beta,\veps)](x,y) & = \exp\left(-\tfrac{1}{\veps} \left[
		c(x,y)-\alpha(x)-\beta(y)
		\right]\right) \cdot \rho(x,y)\,.
\end{align}
\end{subequations}
The second line, \eqref{eqn:StabilizedKernelEntries}, should be used for numerical evaluation such that extreme values in $(\alpha,\beta)$ and $c$ can cancel before exponentiation.
Moreover, we introduce a stabilized version of the $\proxdivSymb$ operator:
\begin{align}
	\label{eqn:StabilizedProxdiv}
	\proxdiv{F}{\veps} : (\sigma,\gamma) \mapsto \ProxKL{F}{\veps}(\exp(-\gamma/\veps)\,\sigma) \oslash \sigma
\end{align}
Note that the regular version of the $\proxdivSymb$ operator, \eqref{eqn:Proxdiv}, is a special case of the stabilized variant with $\gamma=0$.
With $\kernel=\getKernel(\veps)$ and $\kernelStable=\getKernelStable(\alphaBase,\betaBase,\veps)$ we observe that
\begin{subequations}
\label{eqn:StabilizedProxdivEquiv}
\begin{align}
	\proxdiv{F}{\veps}(\kernelStable\,\vRel,\alphaBase) & = \proxdiv{F}{\veps}(\kernel\,v)
		\oslash \exp(\alphaBase/\veps), \\
	\proxdiv{F}{\veps}(\kernelStable^\T \uRel,\betaBase) & = \proxdiv{F}{\veps}(\kernel^\T u)
		\oslash \exp(\betaBase/\veps)\,.
\end{align}
\end{subequations}
For a threshold parameter $\tau>0$ we formally state the stabilized variant of Algorithm \ref{alg:Scaling}.
\begin{algorithmthm}[Stabilized Scaling Algorithm]\hfill
\label{alg:ScalingStabilized}
\begin{algorithmic}[1]
\Function{ScalingAlgorithmStabilized}{$\veps$,$\iterz{\alpha}$,$\iterz{\beta}$}
	\State $(\alphaBase,\betaBase) \gets (\iterz{\alpha},\iterz{\beta})$\SemiCol
		$(\uRel,\vRel) \gets (1_X,1_Y)$\SemiCol
		$\kernelStable \gets \getKernelStable(\alphaBase,\betaBase,\veps)$
	\Repeat
		\While{$[\|\uRel\|_\infty \leq \tau] \wedge [\|\vRel\|_\infty \leq \tau]$}
				\label{alg:ScalingStabilizedTauCheck}
			\State $\uRel \gets \proxdiv{F_X}{\veps}(\kernelStable\,\vRel,\alphaBase)$\SemiCol
				$\vRel \gets \proxdiv{F_Y}{\veps}(\kernelStable^\T \uRel,\betaBase)$
				\Comment{stabilized iteration}
		\EndWhile
		\State $(\alphaBase,\betaBase) \gets (\alphaBase,\betaBase) + \veps \cdot \log(\uRel,\vRel)$\SemiCol
			$(\uRel,\vRel) \gets (1_X,1_Y)$\SemiCol
			$\kernelStable \gets \getKernelStable(\alphaBase,\betaBase,\veps)$
			\Comment{absorption iteration}
	\Until{stopping criterion}
	\State $(\alphaBase,\betaBase) \gets (\alphaBase,\betaBase) + \veps \cdot \log(\uRel,\vRel)$
	\State \Return $(\alphaBase,\betaBase)$
\EndFunction
\end{algorithmic}
\end{algorithmthm}

Any successive combination of stabilized iterations and absorption iterations in Algorithm \ref{alg:ScalingStabilized} is mathematically equivalent to Algorithm \ref{alg:Scaling}, in the sense that they produce the same iterates (keep in mind (\ref{eqn:RelativeScalingFactors}--\ref{eqn:StabilizedProxdivEquiv})).
But numerically, with finite floating point precision, combining both types of iterations can make a significant difference.
In practice one can run several stabilized iterations in a row, occasionally checking whether $(\uRel,\vRel)$ become too large or too small (see line \ref{alg:ScalingStabilizedTauCheck}), and perform an absorption iteration if required.
This inflicts less computational overhead than the direct log-domain formulation \eqref{eqn:SinkhornLogDomain} and largely preserves the simple matrix multiplication structure of the scaling algorithms.

In the definitions for the stabilized kernel, \eqref{eqn:StabilizedKernelEntries}, and $\proxdivSymb$-operator, \eqref{eqn:StabilizedProxdiv}, there still appear exponentials of the form $\exp(\cdot/\veps)$, which may explode as $\veps \to 0$.
Extending the $\max$-argument trick in \eqref{eqn:SinkhornLogDomain} to more general scaling algorithms entails similar questions.
In the examples studied in Section \ref{sec:Numerics} and those given in \cite{ChizatEntropicNumeric2018} we find however, that evaluation of the exponential $\exp(-\gamma/\veps)$ can be avoided. For the special case of standard optimal transport $\veps$ no longer appears in the stabilized step.

\subsection[Epsilon-Scaling]{$\veps$-Scaling}
\label{sec:AlgorithmEpsScaling}

It is empirically and theoretically well-known (cf.~Section \ref{sec:IntroductionRelatedWork}) that convergence of Algorithm \ref{alg:Scaling} becomes slow as $\veps \to 0$. A popular heuristic remedy is the so-called $\veps$-scaling, where one subsequently solves the regularized problem with gradually decreasing values for $\veps$.
Let $\epsList=(\veps_1,\veps_2,\ldots,\veps_n)$ be a list of decreasing positive parameters. We extend Algorithm \ref{alg:ScalingStabilized} as follows:

\begin{algorithmthm}[Scaling Algorithm with $\veps$-Scaling]\hfill
\label{alg:ScalingEpsScaling}
\begin{algorithmic}[1]
\Function{ScalingAlgorithm$\veps$Scaling}{$\epsList$,$\iterz{\alpha}$,$\iterz{\beta}$}
	\State $(\alpha,\beta) \gets (\iterz{\alpha},\iterz{\beta})$
	\For{$\veps \in \epsList$} \Comment{iterate over list, form largest to smallest}
		\State $(\alpha,\beta) \gets$ \Call{ScalingAlgorithmStabilized}{%
			$\veps$,%
			$\alpha$,$\beta$} \label{alg:ScalingEpsScalingCallAlg}
	\EndFor
	\State \Return $(\alpha,\beta)$
\EndFunction
\end{algorithmic}
\end{algorithmthm}
The dual variable $\beta$ is kept constant while changing $\veps$, not the scaling factor $v$, because the optimal dual variables $(\alpha,\beta)$ usually have a stable limit as $\veps \to 0$, while the scaling factors $(u,v)$ diverge (see Sect.~\ref{sec:IntroductionRelatedWork} and also Theorem \ref{thm:AuctionEpsStability}).

So far, very little is known theoretically about the behaviour of $\veps$-scaling for the Sinkhorn algorithm (cf.~Section \ref{sec:IntroductionRelatedWork}).
Empirically, it is shown in Sect.~\ref{sec:NumericsEfficiency} that $\veps$-scaling is highly efficient and the number of required iterations does not increase exponentially. We observe that indeed it behaves similar as in the auction algorithm, as discussed in \cite{YuilleInvisibleHand1994}.
We work towards a theoretical quantification of this in Sect.~\ref{sec:Auction}.

Motivated by this, in practice we recommend a geometric decrease of $\veps$ and choose $\veps_k = \veps_0 \cdot \lambda^k$ such that $\veps_n$ is the desired final value, $\veps_0$ is on the order of the maximal values in the cost function $c$ and $\lambda \in (0,1)$ is a geometric scaling factor, typically in $[0.5,0.75]$.
If $\lambda$ is too small, iterations will start far from convergence after each change of $\veps$, increasing the risk of numerical instabilities and requiring more iterations. On the other hand, if $\lambda$ is too large, many stages of $\veps$-scaling have to be performed, increasing numerical overhead.

\subsection{Kernel Truncation}
\label{sec:AlgorithmSparseKernel}
Storing the dense kernel $\kernel$ and computing dense matrix multiplications during the scaling iterations \eqref{eqn:IterationsScaling} requires a lot of memory and time on large problems.
For several problems with particular structure, remedies have been proposed (Sect.~\ref{sec:IntroductionRelatedWork}).
But these do not comprise non-standard cost functions, as the one used for the Wasserstein-Fisher-Rao distance, \eqref{eqn:WFRCost}. Moreover they are not compatible with the log-stabilization (Section \ref{sec:AlgorithmLogDomain}), thus a certain level of blur cannot be avoided.
We are looking for a more flexible method to accelerate solving.

For many unregularized transport problems the optimal coupling $\pi^\dagger$ is concentrated on a sparse subset of $X \times Y$. In fact, this is the underlying mechanism for the efficiency of most solvers discussed in Section \ref{sec:IntroductionRelatedWork}.
For the regularized problems the optimal coupling will usually be dense. This is due to the diverging derivative of the $\KL$ divergence at zero.
However, as $\veps \to 0$, the optimal coupling quickly converges to an unregularized solution (see Sect.~\ref{sec:IntroductionRelatedWork}, in particular \cite[Thm.~5.8]{Cominetti-ExpBarrierConvergence-1992}). As $\veps \to 0$, large parts of the coupling will approach zero exponentially fast.

So while we will not be able to exactly solve the full problem, by solving suitable sparse sub-problems, we may still expect a reasonable approximation.
We formalize the concept of a sparse sub-problem.
\begin{definition}[Sparse Sub-Problems]
\label{def:Restriction}
Let $F_X$ and $F_Y$ be marginal functions and $c$ be a cost function as in Definition \ref{def:GenericFormulation} and let $\neigh \subset X \times Y$.
We introduce:
\begin{align}
	\label{eqn:SparseCostKernel}
	\costSparse(x,y) & \eqdef \begin{cases}
		c(x,y) & \tn{if } (x,y) \in \neigh\,, \\
		+\infty & \tn{else.}
		\end{cases}
	&
	\kernelSparse(x,y) & \eqdef \begin{cases}
		\kernel(x,y) & \tn{if } (x,y) \in \neigh\,, \\
		0 & \tn{else.}
		\end{cases}
\end{align}
We call problems \eqref{eqn:GenericPrimal} and \eqref{eqn:GenericDual} with $c$ replaced by $\costSparse$ the problems \emph{restricted} to $\neigh$. This corresponds to adding the constraint $\spt \pi \subset \neigh$ to the primal problem, and only enforcing the constraint $\alpha(x) + \beta(y) \leq c(x,y)$ on $(x,y) \in \neigh$ in the dual problem.
The entropy regularized variants of the restricted problems are obtained through replacing $\kernel$ by $\kernelSparse$ in \eqref{eqn:RegularizedGenericPrimal} and \eqref{eqn:RegularizedGenericDual}.
\end{definition}
Clearly, when $\neigh$ is sparse, then so is $\kernelSparse$ and the restricted regularized problem can be solved faster and with less memory.
We now quantify the error inflicted by restriction.
\begin{proposition}[Restricted Kernel and Duality Gap]
	\label{prop:Restriction}
	Let $\veps>0$ and $\neigh \subset X \times Y$. Let $E$ and $J$ be unrestricted regularized primal and dual functionals with kernel $\kernel$, as given in Definition \ref{def:RegularizedGenericFormulation}, and let $\hat{E}$ and $\hat{J}$ be the functionals of the problems restricted to $\neigh$, with sparse kernel $\kernelSparse$ (see Def.~\ref{def:Restriction}).
	
	Further, let $(\alpha,\beta)$ be a pair of dual variables, let $u = \exp(\alpha/\veps)$, $v=\exp(\beta/\veps)$ be the corresponding scaling factors and let $\pi = \diag(u)\,\kernelSparse\,\diag(v)$ be the corresponding (restricted) primal coupling.
	
	Then we find for the primal-dual gap between $\pi$ and $(\alpha,\beta)$:
	\begin{align}
		\label{eqn:TruncatedDualityGap}
		E(\pi) - J(\alpha,\beta) = \hat{E}(\pi) - \hat{J}(\alpha,\beta)
			+ \veps \sum_{(x,y) \in (X \times Y) \setminus \neigh} u(x)\,\kernel(x,y)\,v(y)\,.
	\end{align}
\end{proposition}

\begin{proof}
For the primal score we find:
\begin{align*}
	E(\pi) & = F_X(\projx\,\pi) + F_Y(\projy\,\pi) + \veps \sum_{(x,y) \in X \times Y}
		\left[ \pi(x,y) \log\left(\tfrac{\pi(x,y)}{\kernel(x,y)}\right) - \pi(x,y) + \kernel(x,y) \right]\\
	& = \hat{E}(\pi) + \veps \sum_{(x,y) \in (X \times Y) \setminus \neigh}
		\Big[ \underbrace{
			\pi(x,y) \log\left(\tfrac{\pi(x,y)}{\kernel(x,y)}\right) - \pi(x,y)
			}_{=0} + \kernel(x,y) \Big] \\
	\intertext{Analogously, for the dual score we get:}
	J(\alpha,\beta) & = -F_X^\ast(-\alpha) - F_Y^\ast(-\beta) - \veps \sum_{(x,y) \in X \times Y}
		\kernel(x,y) \cdot \left(\exp([\alpha(x)+\beta(y)]/\veps) - 1\right) \\
	& = \hat{J}(\alpha,\beta) - \veps \sum_{(x,y) \in (X \times Y) \setminus \neigh}
		\kernel(x,y) \cdot \Big( \underbrace{
			\exp([\alpha(x)+\beta(y)]/\veps)
			}_{= u(x)\,v(y)} - 1\Big)
\end{align*}
Together we obtain $E(\pi) - J(\alpha,\beta) = \hat{E}(\pi) - \hat{J}(\alpha,\beta)
		+ \veps \sum_{(x,y) \in (X \times Y) \setminus \neigh} u(x)\,\kernel(x,y)\,v(y)$.
\end{proof}
That is, the primal-dual gap for the original full functionals is equal to the gap for the truncated functionals plus the `mass' that we have chopped off by truncating $\kernel$ to $\kernelSparse$, when using the scaling factors $u$ and $v$.
If some $\neigh$ were known, on which most mass of the optimal $\pi^\dagger$ is concentrated, it would be sufficient to solve the problem restricted to $\neigh$, to get a good approximate solution. The remaining challenge is, how to identify $\neigh$ without knowing $\pi^\dagger$ before.

We propose an iterative re-estimation of $\neigh$, based on current dual iterates and to combine this with the log-stabilized iteration scheme (Section \ref{sec:AlgorithmLogDomain}) and the computation of the stabilized kernel, \eqref{eqn:StabilizedKernelEntries}. For a threshold parameter $\theta>0$ we define the following functions:
\begin{align}
	\label{eqn:getNeigh}
	\getNeigh(\alpha,\beta,\veps, \theta) & \eqdef \{(x,y) \in X \times Y \,\colon\, \exp(-\tfrac{1}{\veps}[c(x,y)-\alpha(x) - \beta(y)]) \geq \theta \} \\
	\label{eqn:getKernelSparse}
	[\getKernelSparse(\alpha, \beta, \veps, \theta)](x,y) & \eqdef \begin{cases}
		\exp(-\tfrac{1}{\veps}[c(x,y)-\alpha(x)-\beta(y)])\,\rho(x,y) & \tn{if }
			(x,y) \in \getNeigh(\alpha,\beta,\veps,\theta)\,, \\
		0 & \tn{else.}
		\end{cases}
\end{align}
$\getKernelSparse$ can be used instead of $\getKernelStable$ in Algorithm \ref{alg:ScalingStabilized}.
We refer to this as \emph{absorption iteration with truncation}.
For this combination one finds a simple bound for the primal-dual gap comparison of Proposition \ref{prop:Restriction}.
\begin{proposition}[Simple Duality Gap Estimate for Absorption Iterations with Truncation]
	\label{prop:RestrictionSimple}
	For a regularized problem as in Definition \ref{def:RegularizedGenericFormulation} with functionals $E$ and $J$, let $(u,v)$ be a pair of diagonal scaling factors and $(\alpha,\beta) = \veps \cdot \log(u,v)$, let $(\alphaBase,\betaBase)$ a pair of dual variables and $(\uRel,\vRel)$ a pair of relative scaling factors such that $u = \uRel \cdot \exp(\alphaBase/\veps)$ and $v = \vRel \cdot \exp(\betaBase/\veps)$.

	Let further
	$\neigh = \getNeigh(\alphaBase,\betaBase,\veps,\theta)$, 
	$\kernelBase = \getKernelSparse(\alphaBase,\betaBase,\veps,\theta)$, let $\hat{E}$ and $\hat{J}$ be the functionals restricted to $\neigh$ (see Def.~\ref{def:Restriction}) and $\pi = \diag(\uRel)\,\kernelBase\,\diag(\vRel)$. Then $E(\pi) - J(\alpha,\beta) \leq \hat{E}(\pi) - \hat{J}(\alpha,\beta) + \|\uRel\|_\infty \cdot \|\vRel\|_\infty \cdot \theta \cdot \rho(X \times Y)$.
\end{proposition}
\begin{proof}
	By virtue of Proposition \ref{prop:Restriction}
	\begin{align*}
		E(\pi) - J(\alpha,\beta) = \hat{E}(\pi) - \hat{J}(\alpha,\beta) + \sum_{(x,y) \in (X \times Y) \setminus \neigh} u(x)\,\kernel(x,y)\,v(y)\,.
	\end{align*}
	For $(x,y) \in (X \times Y) \setminus \neigh$ one has $\exp(-\tfrac{1}{\veps}[c(x,y)-\alphaBase(x)-\betaBase(y)]) < \theta$ and therefore
	\begin{align*}
		u(x)\,\kernel(x,y)\,v(y) & = \uRel(x)\,\exp\left(-\tfrac{1}{\veps}[
			c(x,y)-\alphaBase(x)-\betaBase(y)]\right) \cdot \rho(x,y) \cdot \vRel(y) \\
		& \leq \uRel(x)\,\vRel(y)\,\theta\,\rho(x,y)\,.
	\end{align*}
	The result follows by bounding $\uRel(x) \leq \|\uRel\|_\infty$, $\vRel(y) \leq \|\vRel\|_\infty$ and summing over $(X \times Y) \setminus \neigh$.
\end{proof}
This implies that in Algorithm \ref{alg:ScalingStabilized} with truncation the additional duality gap error due to the sparse kernel is bounded by $\|\iterl{\uRel}\|_\infty \cdot \|\iterl{\vRel}\|_\infty \cdot \theta \cdot \rho(X \times Y)$.
In particular, before every stabilized iteration the error is bounded by $\tau^2 \cdot \theta \cdot \rho(X \times Y)$ and after every absorption iteration it is bounded by $\theta \cdot \rho(X \times Y)$.
This bound is easy to evaluate and does not require to sum over $(X \times Y) \setminus \neigh$, as the exact expression in Proposition \ref{prop:Restriction}. We find that in practice this truncation error bound can be kept much smaller than the remaining primal-dual gap $\hat{E}(\pi) - \hat{J}(\alpha,\beta)$.

In general the stabilized iteration scheme \emph{with truncation} might not converge. However, by Proposition \ref{prop:RestrictionSimple}, if one regularly performs an absorption iteration before $\|\iterl{\uRel}\|_\infty \cdot \|\iterl{\vRel}\|_\infty$ becomes too large, the potential oscillations in the primal iterates and primal and dual functionals are numerically negligible.

\subsection{Multi-Scale Scheme}
\label{sec:AlgorithmMultiScale}

Finally, we propose to combine the stabilized sparse iterations with a hierarchical multi-scale scheme, analogous to the ideas in \cite{MultiscaleTransport2011,SchmitzerSchnoerr-SSVM2013,ObermanOptimalTransportationLP2015}.

This serves two purposes:
First, a hierarchical representation of the problem allows to determine the truncated sparse stabilized kernel $\getKernelSparse$, \eqref{eqn:getKernelSparse}, with a coarse-to-fine tree search, without explicitly testing all pairs $(x,y) \in X \times Y$.
The second reason is to make the combination of $\veps$-scaling (Algorithm \ref{alg:ScalingEpsScaling}) with the truncated stabilized scheme more efficient.
For a fixed threshold $\theta$, while $\veps$ is large, the support of the truncated kernel $\getKernelSparse$ will contain many variables.
At the same time, due to the blur induced by the regularization, the primal iterates will not provide a sharply resolved assignment.
Solving the problems with large $\veps$-value on a coarser grid reduces the number of required variables, without losing much spatial accuracy.
As $\veps$ decreases, so does the number of variables in $\getKernelSparse$ (since the exponential function decreases faster), and the resolution of $X$ and $Y$ can be increased.
Therefore, it is reasonable to coordinate the reduction of $\veps$ with increasing the spatial resolution of the transport problem, until the desired regularization and resolution are attained.

We will now briefly recall the hierarchical representation of a transport problem from \cite{SchmitzerSchnoerr-SSVM2013}.

\begin{definition}[Hierarchical Partition and Multi-Scale Measure Approximation \cite{SchmitzerSchnoerr-SSVM2013}]
	\label{def:HierarchicalPartition}
	For a discrete set $X$ a \emph{hierarchical partition} is an ordered tuple $(\hpartX_0,\ldots,\hpartX_I)$ of partitions of $X$ where $\hpartX_0 = \{ \{x\} \colon x \in X\}$ is the trivial partition of $X$ into singletons and each subsequent level is generated by merging cells from the previous level, i.e.~for $i \in \{1,\ldots,I\}$ and any $\hcellX \in \hpartX_i$ there exists some $\hat{\hpartX} \subset \hpartX_{i-1}$ such that $\hcellX = \bigcup_{\hat{\hcellX} \in \hat{\hpartX}} \hat{\hcellX}$.
	For simplicity we assume that the coarsest level is the trivial partition into one set: $\hpartX_I = \{X\}$.
	We call $I>0$ the \emph{depth} of $\hpartX$.
	
	This implies a directed tree graph with vertex set $\bigcup_{i=0}^I \hpartX_{i}$. For $i$, $j \in \{0,\ldots,I\}$, $i<j$ we say $\hcellX \in \hpartX_{i}$ is a \emph{descendant} of $\hcellX' \in \hpartX_j$ when $\hcellX \subset \hcellX'$. We call $\hcellX$ a \emph{child} of $\hcellX'$ for $i=j-1$, and a \emph{leaf} for $i=0$.
	For some $\mu \in \R^X$ its \emph{multi-scale measure approximation} is the tuple $(\mu_0,\ldots,\mu_I)$ of measures $\mu_i \in \R^{\hpartX_i}$ defined by $\mu_i(\hat{\hpartX}) = \mu(\bigcup_{\hcellX \in \hat{\hpartX}} \hcellX)$ for all subsets $\hat{\hpartX} \subset \hpartX_i$ and $i=0,\ldots I$.
	For convenience we often identify $X$ with the finest partition level $\hpartX_0$ and $\mu$ with $\mu_0$.
\end{definition}

\begin{definition}[Hierarchical Dual Variables and Costs \cite{SchmitzerSchnoerr-SSVM2013}]
\label{def:HierarchicalDualsCosts}
Let $X$ and $Y$ be discrete sets with hierarchical partitions $\hpartX=(\hpartX_0,\ldots,\hpartX_I)$, $\hpartY=(\hpartY_0,\ldots,\hpartY_I)$ of depth $I$, let $\alpha \in \R^{X}$ and $\beta \in \R^{Y}$ be functions over $X$ and $Y$, and let $c \in \R^{X \times Y}$ be a cost function.

Then we define the extension $\hat{\alpha}=(\hat{\alpha}_0,\ldots,\hat{\alpha}_I)$ of $\alpha$ onto the full partition $\hpartX$ by
\begin{align}
	\label{eqn:HierarchicalDuals}
	\hat{\alpha}_i(\hcellX) = \max_{x \in \hcellX} \alpha(x) = 
		\begin{cases}
			\alpha(x) & \tn{if } i=0 \tn{ and } \hcellX = \{x\} \tn{ for some } x \in X, \\
			\max_{\hcellX' \in \children(\hcellX)} \hat{\alpha}_{i-1}(\hcellX') & \tn{if } i>0,
		\end{cases}
\end{align}
for $i \in \{0,\ldots,I\}$ and $\hcellX \in \hpartX_i$ and analogous for $\hat{\beta}$ and $\beta$.
Similarly, define an extension $\hat{c}$ of $c$ by
\begin{align}
	\label{eqn:HierarchicalCost}
	\hat{c}_i(\hcellX,\hcellY) = \min_{(x,y) \in \hcellX \times \hcellY} c(x,y)
\end{align}
for $i \in \{0,\ldots,I\}$, $\hcellX \in \hpartX_i$ and $\hcellY \in \hpartY_i$.
\end{definition}

For $i \in \{0,\ldots,I\}$, $x \in \hcellX \in \hpartX_i$, $y \in \hcellY \in \hpartY_i$ we find
\begin{align}
	\label{eqn:HierarchicalConstraint}
	\hat{c}_i(\hcellX,\hcellY)-\hat{\alpha}_i(\hcellX)-\hat{\beta}_i(\hcellY)
	\leq c(x,y) - \alpha(x) - \beta(y)\,.
\end{align}
Now we can implement a hierarchical tree-search for $\getNeigh$ (and analogously $\getKernelSparse$).

\begin{algorithmthm}[Hierarchical Search for $\getNeigh$]\hfill
\label{alg:HierarchicalTruncation}
\begin{algorithmic}[1]
\Function{$\getNeigh$}{$\alpha$,$\beta$,$\veps$,$\theta$}
	\State $(\hat{\alpha},\hat{\beta}) \gets \tn{hierarchical extensions of } (\alpha,\beta)$
		\Comment{see \eqref{eqn:HierarchicalDuals}}
	\State $\neigh \gets$ \Call{ScallCell}{%
		$\hat{\alpha}$,$\hat{\beta}$,$\veps$,$\theta$,$I$,$\{X\}$,$\{Y\}$} %
		\Comment{call on coarsest partition level}
\EndFunction
\Statex
\Function{ScanCell}{$\hat{\alpha}$,$\hat{\beta}$,$\veps$,$\theta$,$i$,$\hcellX$,$\hcellY$}
	\State $\neigh' \gets \emptyset$ \Comment{temporary variable for result}
	\If{$\hat{c}_i(\hcellX,\hcellY)-\hat{\alpha}_i(\hcellX) - \hat{\beta}_i(\hcellY) %
			\leq -\veps \cdot \log \theta$} %
			\Comment{if cell cannot be ruled out at this level}
		\If{$i>0$} \Comment{if not yet at finest level, check on all children}
			\For{$(\hcellX',\hcellY') \in \tn{children}(\hcellX) \times \tn{children}(\hcellY)$}
				\State $\neigh' \gets \neigh' \cup$ %
						\Call{ScanCell}{%
						$\hat{\alpha}$,$\hat{\beta}$,$\veps$,$\theta$,$i-1$,$\hcellX'$,$\hcellY'$}
			\EndFor
		\Else \Comment{if at finest level, add variable}
			\State $\neigh' \gets \neigh' \cup (\hcellX \times \hcellY)$
				\Comment{recall $\hcellX=\{x\}$, $\hcellY=\{y\}$ for some $(x,y) \in X \times Y$ at $i=0$}
		\EndIf
	\EndIf
	\State \Return $\neigh'$
\EndFunction
\end{algorithmic}
\end{algorithmthm}

From \eqref{eqn:HierarchicalConstraint} follows directly that Algorithm \ref{alg:HierarchicalTruncation} implements \eqref{eqn:getNeigh}.

In many applications the discrete sets $X$ and $Y$ are point clouds in $\R^d$ and the hierarchical partitions are $2^d$-trees over $X$ and $Y$ (see e.g.~\cite{SchmitzerShortCuts2015}).
The cost function $c$ is often originally defined on the whole product space $\R^d \times \R^d$ (such as the squared Euclidean distance).
For the validity of Algorithm \ref{alg:HierarchicalTruncation} it suffices if $\hat{c}_i(\hcellX,\hcellY) \leq \min_{(x,y) \in \hcellX \times \hcellY} c(x,y)$.
This allows to avoid computing (and storing) the full cost matrix $c \in \R^{X \times Y}$ and the explicit minimizations in \eqref{eqn:HierarchicalCost}. $c$ and lower bounds on $\hat{c}_i$ can be computed on-demand directly using the tree-structure.

\begin{figure}[hbt]
	\centering
	\includegraphics{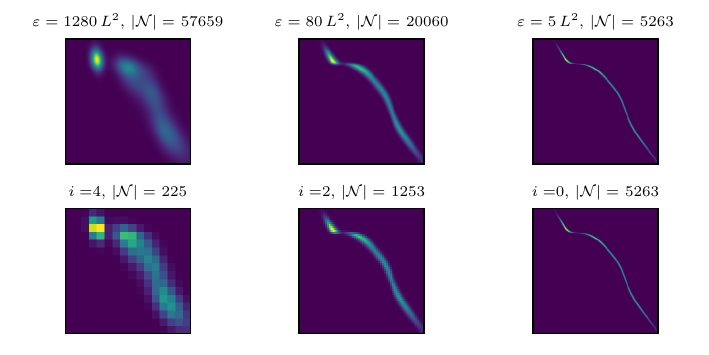}
	\caption{$\veps$-scaling, truncated kernels and multi-scale scheme. %
	$X=Y$ is a uniform one-dimensional grid, representing $[0,1]$, $|X|=256$, $h=256^{-1}$. $\mu$ and $\nu$ are smooth mixtures of Gaussians.
	\textbf{Top row} Density of optimal coupling $\pi^\dagger$ on $X^2$ for various $\veps$. $|\neigh|$ is the number of variables in the truncated, stabilized kernel for fixed $\theta=10^{-10}$. %
	As $\veps$ decreases, so does $|\neigh|$, since $\pi^\dagger$ becomes more concentrated.
	\textbf{Bottom row} Optimal couplings for same $\veps$ as top row, but for different levels $i$ of hierarchical partitions. $i$ and $\veps$ were chosen to keep number of variables per $x \in X$ approximately constant. For high $\veps$ (and $i$) $|\neigh|$ is now dramatically lower. While $\pi^\dagger$ is `pixelated' for high $i$, due to blur, it provides roughly the same spatial information as the top row. %
	Images in third column are identical.
	}
	\label{fig:MultiScaleEpsScaling}
\end{figure}

The second purpose of the multi-scale scheme is the combination with $\veps$-scaling.
As explained above, the purpose is to reduce the number of variables while $\veps$ is large. For an illustration see Fig.~\ref{fig:MultiScaleEpsScaling}.
For this, we divide the list $\epsList$ of regularization parameters $\veps$ into multiple lists $(\epsList_0,\ldots,\epsList_I)$, with the largest values in $\epsList_I$ and the smallest (and final) values in $\epsList_0$, and sorted from largest to smallest within each $\epsList_i$. Then, for every $i$ from $I$ down to $0$ we perform $\veps$-scaling with list $\epsList_i$ at hierarchical level $i$, using the dual solution at each level as initialization at the next stage. The full algorithm, combining $\log$-stabilization, $\veps$-scaling, kernel truncation and the multi-scale scheme, is sketched next.

\begin{algorithmthm}[Full Algorithm]\hfill
\label{alg:Full}
\begin{algorithmic}[1]
\Function{ScalingAlgorithmFull}{$(\epsList_0,\ldots,\epsList_I)$,$\theta$}
	\State $i=I$\SemiCol
		$(\alpha,\beta) \gets ((0),(0))$ \Comment{initialize scale counter and dual variables}
	\While{$i \geq 0$}
		\CommentLine{solve problem at scale $i$ with $\veps$-scaling over $\epsList_i$}
		\For{$\veps \in \epsList_i$} \Comment{iterate over list, from largest to smallest}
			\State $(\alpha,\beta) \gets$ \Call{ScalingAlgorithmStabilized}{$i$,$\veps$,$\theta$,$\alpha$,$\beta$}
				\label{alg:FullStableCall}
		\EndFor
		\State $i \gets i-1$
		\If{$i\geq 0$} \Comment{refine dual variables}
			\State $(\alpha,\beta) \gets$ \Call{RefineDuals}{$i$,$\alpha$,$\beta$}
		\EndIf
	\EndWhile
	\State \Return $(\alpha,\beta)$
\EndFunction
\end{algorithmic}
Note: \textsc{ScalingAlgorithmStabilized} refers to calling Algorithm \ref{alg:ScalingStabilized} for solving the problem at scale $i$, with $\getKernelStable$ replaced by $\getKernelSparse$, \eqref{eqn:getKernelSparse}, with threshold $\theta$, implemented according to Algorithm \ref{alg:HierarchicalTruncation}. Accordingly, two arguments $i$ and $\theta$ were added.
\textsc{RefineDuals} initializes the dual variables $(\alpha,\beta)$ at level $i$ by setting the values at $\hcellX$ to the previous values at $\parent(\hcellX)$ for all cells $\hcellX$ in $\hpartX_i$.
\end{algorithmthm}

\begin{remark}[Hierarchical Representation of $F_X$, $F_Y$]
	To solve the problem at hierarchical scale $i$, not only do we need a coarse version of $c$, as given in \eqref{eqn:HierarchicalCost}. In addition we need hierarchical versions of the marginal functions $F_X$, $F_Y$, see \eqref{eqn:RegularizedGeneric}.
	An appropriate choice is often clear from the context of the problem. For example, for an optimal transport problem between $\mu$ and $\nu$, see Def.~\ref{def:EntropicOptimalTransport}, we set $F_{{\hpartX}_i} = \iota_{\{\mu_i\}}$, where $\mu_i$ is taken from the multi-scale measure approximation of $\mu$ (see Def.~\ref{def:HierarchicalPartition}).
	For the unbalanced transport problem with $\KL$ fidelity, Def.~\ref{def:UnbalancedTransport}, we use $F_{\hpartX_i} = \lambda \cdot \KL_{\hpartX_i}(\cdot|\mu_i)$.
\end{remark}

This completes the modifications of the diagonal scaling algorithm. Their usefulness will be demonstrated numerically in Sect.~\ref{sec:Numerics}.


\newcommand{\qTarget}{q_{\tn{target}}}
\newcommand{\cMax}{C}
\newcommand{\edges}{\mc{E}}
\newcommand{\Jdiag}{\hat{J}}
\newcommand{\Jdual}{J}
\newcommand{\JdiagEff}{\hat{\mc{J}}}
\newcommand{\nK}{R}
\section{Analogy between Sinkhorn and Auction Algorithm}
\label{sec:Auction}
In this section we develop a new complexity analysis of the Sinkhorn algorithm and examine the efficiency of $\veps$-scaling.
In \cite{YuilleInvisibleHand1994} an intuitive similarity between the Sinkhorn algorithm for the entropy regularized linear assignment problem and the auction algorithm was pointed out.
This similarity motivates our approach.

In this section we only consider the standard Sinkhorn algorithm (as opposed to general scaling algorithms), since the auction algorithm solves the linear assignment problem and assumptions on fixed marginals $\mu$, $\nu$ are required for our analysis.

The auction algorithm is briefly recalled in Section \ref{sec:AuctionRecall}.
In Section \ref{sec:AuctionSimple} we introduce an asymmetric variant of the Sinkhorn algorithm, that is more similar to the original auction algorithm and provide an analogous worst-case estimate for the number of iterations until a given precision is achieved.
A stability result for the dual optimal solutions under change of the regularization parameter $\veps$ is given in Section \ref{sec:AuctionStability} and we discuss how it relates to $\veps$-scaling in Sect.~\ref{sec:AuctionEpsScaling}.

\subsection{Auction Algorithm}
\label{sec:AuctionRecall}
For the sake of self-containedness, in this section we briefly recall the auction algorithm and its basic properties. Note that compared to the original presentation (e.g.~\cite{Bertsekas-ParallelAuction1988}) we flipped the overall sign for compatibility with the notion of optimal transport.

In the following we consider a linear assignment problem, i.e.~an optimal transport problem between two discrete sets $X$, $Y$ with equal cardinality $|X|=|Y|=N$ where the marginals $\mu \in \R_+^X$, $\nu \in \R_+^Y$ are the counting measures.
For simplicity we assume that the cost function $c \in \R_+^{X \times Y}$ is finite and non-negative.

The main loop of the auction algorithm is divided into two parts: During the bidding phase, elements of $X$ that are unassigned determine their locally most attractive counterpart in $Y$ (taking into account the current dual variables) and submit a bid for them.
During the assignment phase, all elements of $Y$ that received at least one bid, pick the most attractive one and change the current assignment accordingly.
A formal description is given in the following.

\begin{algorithmthm}[Auction Algorithm]\hfill
\label{alg:Auction}
\begin{algorithmic}[1]
\Function{AuctionAlgorithm}{$\iterz{\beta}$}
	\State $\pi \gets 0_{X \times Y}$\SemiCol
		$\beta \gets \iterz{\beta}$
		\Comment{initialize variables: `empty' primal coupling, zero dual variable}
	\While{$\pi(X \times Y) < N$}
		\State $B(y) \gets \emptyset$ for all $y \in Y$ \Comment{start bidding phase: initialize empty bid lists}
		\For{$x \in \{x' \in X : \pi(\{x'\} \times Y) = 0\}$} \Comment{iterate over unassigned $x$}
			\State $y \gets \argmin_{y' \in Y} [c(x,y')-\beta(y')]$ \Comment{pick some element from argmin} \label{alg:Auction:FindYBid}
			\State $\alpha(x) \gets c(x,y)-\beta(y)$
				\Comment{set dual variable}
				\label{alg:Auction:Alpha}
			\State $B(y) \gets B(y) \cup \{x\}$ \Comment{submit bid to $y$, i.e.~add $x$ to bid list of $y$}
		\EndFor
		\For{$y \in \{y' \in Y : B(y') \neq \emptyset\}$}
			\Comment{assignment phase: iterate over all $y$ that received bids}
			\State $\pi(\cdot,y) \gets 0$ \Comment{set column of coupling to zero}
			\State $x \gets \argmin_{x' \in B(y)} [c(x',y)-\alpha(x')]$ \Comment{find best bidder, pick one if multiple}
			\State $\beta(y) \gets c(x,y) - \alpha(x) - \veps$\SemiCol
				$\pi(x,y) \gets 1$
				\Comment{update dual variable and coupling}
				\label{alg:Auction:Beta}
		\EndFor
	\EndWhile
	\State \Return $(\pi,(\alpha,\beta))$
\EndFunction
\end{algorithmic}
\end{algorithmthm}
\begin{remark}
	In the above algorithm, line \ref{alg:Auction:Alpha} is usually replaced by $\alpha(x) \gets \min_{y' \in Y \setminus \{y\}} [ c(x,y')-\beta(y')]$, which in practice may reduce the number of iterations. It does not affect the following worst-case analysis however, therefore we keep the simpler version.
\end{remark}

We briefly summarize the main properties of the algorithm.
\begin{proposition}
	\label{prop:AuctionMonotonicity}
	With $\veps>0$ and $\iterz{\beta}=0_Y$, Algorithm \ref{alg:Auction} has the following properties:
	\begin{enumerate}[(i)]
		\item $\alpha$ is increasing, $\beta$ is decreasing.
		\item After each assignment phase one finds $\alpha(x) + \beta(y) \leq c(x,y)$ and $[\pi(x,y) > 0]$ $\Rightarrow$ $[\alpha(x) + \beta(y) \geq c(x,y) - \veps]$.
			The latter property is called $\veps$-complementary slackness.
			\label{item:AuctionPDSlack}
		\item The primal iterate satisfies $\projx \pi \leq \mu$ and $\projy \pi \leq \nu$.
			\label{item:AuctionPrimalSubfeasible}
		\item The algorithm terminates after at most $N \cdot (\cMax/\veps+1)$ iterations, where $\cMax = \max c$.
	\end{enumerate}
\end{proposition}
For a proof see for example \cite{Bertsekas-ParallelAuction1988}.
From $\veps$-complementary slackness we deduce the following result.
\begin{corollary}
	\label{cor:AuctionPDGap}
	Upon convergence, the primal-dual gap of $\pi$ and $(\alpha,\beta)$, cf.~Def.~\ref{def:OptimalTransport}, is bounded by $\la c, \pi \ra - (\la \mu,\alpha \ra + \la \nu, \beta \ra) \leq N \cdot \veps$.
	If $c$ is integer and $\veps<1/N$, then the final primal coupling is optimal.
\end{corollary}

\begin{remark}[$\veps$-Scaling for the Auction Algorithm]
\label{rem:AuctionEpsScaling}
During the auction algorithm it may happen that several elements in $X$ compete for the same target $y \in Y$, leading to the minimal decrease of $\beta(y)$ by $\veps$ in each iteration. This phenomenon has been dubbed `price haggling' \cite{Bertsekas-MinCostFlowReview1988} and can cause poor practical performance of the algorithm, close to the worst-case iteration bound.
The impact of price haggling can be reduced by the $\veps$-scaling technique, where the algorithm is successively run with a sequence of decreasing values for $\veps$, each time using the final value of $\beta$ as initialization of the next run (see also Algorithm \ref{alg:ScalingEpsScaling}).
With this technique the factor $C/\veps$ in the iteration bound can essentially be reduced to a factor $\log(C/\veps)$. An analysis of the $\veps$-scaling technique for more general min-cost-flow problems can be found in \cite{Bertsekas-MinCostFlowReview1988}.
\end{remark}

\subsection{Asymmetric Sinkhorn Algorithm and Iteration Bound}
\label{sec:AuctionSimple}
We now introduce a slightly modified variant of the standard Sinkhorn algorithm, derive an iteration bound and make a comparison with the auction algorithm. We emphasize that this modification is primarily made to facilitate theoretical study of the algorithm and to understand why convergence becomes slow as $\veps \to 0$. We do not advocate its merits in an actual implementation.

For $\mu \in \prob(X)$, $\nu \in \prob(Y)$ and a cost function $c \in \R_+^{X \times Y}$ we consider the entropic optimal transport problem (Def.~\ref{def:EntropicOptimalTransport}).
Set the reference measure $\rho$ for regularization, see \eqref{eqn:Kernel}, to the product measure $\rho(x,y) = \mu(x) \cdot \nu(y)$.
We state the modified Sinkhorn algorithm with parameter $\qTarget \in (0,1)$ that measures how much mass has to be assigned.
\begin{algorithmthm}[Asymmetric Sinkhorn Algorithm]\hfill
\label{alg:SinkhornAsymmetric}
\begin{algorithmic}[1]
\Function{AsymmetricSinkhorn}{$\veps$,$\iterz{v}$,$\qTarget$}
	\State $K \gets \getKernel(\veps)$\SemiCol
		$v=\iterz{v}$
		\Comment{compute kernel, initialize scaling factor}
	\Repeat
		\State $u \gets \mu \oslash (\kernel\,v)$\SemiCol
			$\hat{v} \gets \nu \oslash (\kernel^\T u)$
			\label{alg:SinkhornAsymmetric:uvHat}
		\State $v \gets \min\{v,\hat{v}\}$ \Comment{element-wise minimum}
			\label{alg:SinkhornAsymmetric:vMin}
		\State $\pi \gets \diag(u)\,\kernel\,\diag(v)$\SemiCol
			$q \gets \pi(X \times Y)$ \Comment{update coupling and `assigned mass'-fraction $q$}
			\label{alg:SinkhornAsymmetric:pi}
	\Until{$q \geq \qTarget$}
	\State \Return $(\pi,(u,v))$
\EndFunction
\end{algorithmic}
\end{algorithmthm}

The only differences to the standard Sinkhorn algorithm (given by Algorithm \ref{alg:Scaling} with $\proxdivSymb$-operators \eqref{eqn:ProxdivSinkhorn}) lie in line \ref{alg:SinkhornAsymmetric:vMin} and in the choice of the specific stopping criterion (see Remark \ref{rem:SinkhornAsymmetricStoppingCriterion} for a discussion).
In the standard algorithm one would set $v \gets \hat{v}$.
The modification implies that $v$ is monotonously decreasing, which implies the following result for Algorithm \ref{alg:SinkhornAsymmetric} in the spirit of Proposition \ref{prop:AuctionMonotonicity}. This monotonicity is crucial for bounding the number of iterations (see also Remark \ref{rem:AnalogyAuction}).

Throughout this section, for clarity, we enumerate the iterates $u$, $v$, as well as the auxiliary variables $\hat{v}$, $\pi$ and $q$ in Algorithm \ref{alg:SinkhornAsymmetric}, starting with $\iterz{v}$ and proceeding with $(\iter{u}{1},\iter{v}{1},\iter{\hat{v}}{1},\iter{\pi}{1},\iter{q}{1}),\ldots$, similar to formulas \eqref{eqn:IterationsScaling} and the corresponding dual variable iterates $(\iterl{\alpha},\iterl{\beta},\iterl{\hat{\beta}})\allowbreak=\allowbreak\veps \cdot \log(\iterl{u},\iterl{v},\iterl{\hat{v}})$.
\begin{proposition}[Monotonicity of Asymmetric Sinkhorn Algorithm]\hfill
	\label{prop:SinkhornAsymmetricMonotonicity}
	\begin{enumerate}[(i)]
		\item $u$ and $\alpha=\veps\,\log u$ are increasing, $v$ and $\beta=\veps\,\log v$ are decreasing, $q$ is increasing.
		\item $\projx\,\pi \leq \mu$ and $\projy\,\pi \leq \nu$. We say $\pi$ is \emph{sub-feasible}.
		\item There exists some $y^\ast \in Y$ such that $v(y^\ast) = \iterz{v}(y^\ast)$ for all iterations.	
	\end{enumerate}
\end{proposition}
\begin{proof}
	By construction we have $\iterll{v} \leq \iterl{v}$. Consequently $\kernel\,\iterll{v} \leq \kernel\,\iterl{v}$ and thus $\iterll{u} \geq \iterl{u}$ and eventually $\iterll{\hat{v}} \leq \iterl{\hat{v}}$.
	
	After updating $\iterll{u}$ the row constraints are satisfied. That is $\projx \diag(\iterll{u})\,\kernel\,\diag(\iterl{v}) = \mu$. Since $\iterll{v}$ is only decreased (i.e.~if the corresponding column constraint is violated from above), afterwards the iterate $\iterll{\pi}$ is sub-feasible.
	
	Since $\iterl{\hat{v}}$ is decreasing, it follows that if $\iterl{v}(y) = \iterl{\hat{v}}(y)$ for some $y \in Y$, then $\iter{v}{k}(y) = \iter{\hat{v}}{k}(y)$ for all $k \geq \ell$. Let $\iterl{Y} = \{ y \in Y \,:\, \iterl{v}(y) = \iterl{\hat{v}}(y)\}$. Then $\iterl{Y} \subset \iterll{Y}$. Conversely, if $y \notin \iterl{Y}$, then $\iterl{v}(y) < \iterl{\hat{v}}(y)$ and therefore $\iterl{v}(y) = \iterz{v}(y)$.
	
	Let now $\iterl{q}(y) \eqdef \iterl{v}(y)\,[\kernel^\T \iterl{u}](y)$. If $y \in \iterll{Y}$, then $\iterll{q}(y) = \nu(y) \geq \iterl{q}(y)$ (as $\iterl{q}(y)$ can never exceed $\nu(y)$). If $y \notin \iterll{Y}$, then $\iterll{v}(y) = \iterl{v}(y) = \iterz{v}(y)$ and since $\iterl{u}$ is increasing, we find $\iterll{q}(y) \geq \iterl{q}(y)$.
	We obtain $\iterll{q} = \sum_{y \in Y} \iterll{q}(y) \geq \iterl{q}$.
	
	When $\iterl{Y} \neq Y$, there exists some $y^\ast \in Y$ with $y^\ast \notin \iter{Y}{k}$, $\iter{v}{k}(y^\ast) = \iterz{v}(y^\ast)$ for $k \in \{1,\ldots,\ell\}$.
	If $\iterl{Y} = Y$, then $\iterl{\hat{v}} = \iterl{v} \leq \iter{v}{\ell-1}$. By construction one has $(\iterl{u})^\T\,\kernel\,\iter{v}{\ell-1}=\mu(X)$ and $(\iterl{u})^\T\,\kernel\,\iterl{\hat{v}}=\nu(Y)=\mu(X)$. So if $\iterl{Y}=Y$, in fact $\iterl{v} = \iter{v}{\ell-1}$.
	Consequently, there exists some $y^\ast \in Y$ with $\iterl{v}(y^\ast) = \iterz{v}(y^\ast)$ for all iterations.
\end{proof}
Let us further investigate the increments of the dual variable iterates $\iterl{\alpha}=\veps\,\log(\iterl{u})$.
\begin{lemma}[Minimal Increment of $\iterl{\alpha}$]
	\label{lem:SinkhornAsymmetricAlphaIncrement}
	For $\ell \geq 1$ have
	$\la \iterll{\alpha} - \iterl{\alpha}, \mu \ra \geq \veps\,(1-\iterl{q})$.
\end{lemma}
\begin{proof}
	Recall that $\iterl{\pi} = \diag(\iterl{u})\,\kernel\,\diag(\iterl{v})$,
	and introduce $\iterl{\pi'} = \diag(\iterll{u})\,\kernel\,\diag(\iterl{v})$.
	Consider the following evaluations of the dual functional:
	\begin{align*}
		J(\iterl{\alpha},\iterl{\beta}) & = \langle \iterl{\alpha}, \mu \rangle + \langle \iterl{\beta}, \nu \rangle
			- \veps \cdot \iterl{\pi}(X \times Y) + \veps \cdot \kernel(X \times Y)\\
		J(\iterll{\alpha},\iterl{\beta}) & = \langle \iterll{\alpha}, \mu \rangle + \langle \iterl{\beta}, \nu \rangle
			- \veps \cdot \iterl{\pi'}(X \times Y) + \veps \cdot \kernel(X \times Y)
	\end{align*}
	Note that $\iterl{\pi}(X \times Y) = \iterl{q}$, $\iterl{\pi'}(X \times Y) = 1$ and since going from $\iterl{\alpha}$ to $\iterll{\alpha}$ corresponds to a block-wise dual maximization have $J(\iterll{\alpha},\iterl{\beta}) \geq J(\iterl{\alpha},\iterl{\beta})$. The claim follows.
\end{proof}
With these tools we can bound the total number of iterations to reach a given precision.
\begin{proposition}[Iteration Bound for the Asymmetric Sinkhorn Algorithm]
	\label{prop:SinkhornAsymmetricIterationBound}
	Initializing with $\iterz{\beta} = 0_Y$ $\Leftrightarrow$ $\iterz{v} = 1_Y$, for a given $\qTarget \in (0,1)$ the number of iterations $n$ necessary to achieve $\iter{q}{n} \geq \qTarget$ is bounded by
	$n \leq 2+ \frac{\cMax}{\veps \cdot (1-\qTarget)}$
	where $\cMax=\max c$.
	Moreover, $\la \iterl{u}, \mu \ra \leq \exp(\cMax/\veps)$ for all iterates $\ell \geq 1$.
\end{proposition}

\begin{proof}
	Let us look at the first `bid' $\itero{\alpha}$. With $c \geq 0$ we have
	\begin{align}
		\label{eqn:SinkhornAsymmetricIterationsBoundProofAlpha1}
		\itero{\alpha}(x) & = \veps\,\log \left( \frac{\mu(x)}{[\kernel\,\iterz{v}](x)} \right)
			= \veps\,\log \left( \frac{1}{\sum_{y \in Y} \nu(y) \exp(-c(x,y)/\veps)} \right)
			\geq \veps\,\log \left( \frac{1}{\sum_{y \in Y} \nu(y)} \right) = 0.
	\end{align}
	By virtue of Proposition \ref{prop:SinkhornAsymmetricMonotonicity} get $\iterl{q} \leq \iter{q}{n}$ for $\ell \leq n$. With Lemma \ref{lem:SinkhornAsymmetricAlphaIncrement} this implies
	$\langle \iter{\alpha}{n} - \itero{\alpha}, \mu \rangle \geq \sum_{\ell=1}^{n-1} \veps \cdot (1-\iterl{q}) \geq \veps \cdot (n-1) \cdot (1-\iter{q}{n})$
	and with \eqref{eqn:SinkhornAsymmetricIterationsBoundProofAlpha1}
	\begin{align}
		\label{eqn:SinkhornAsymmetricIterationsBoundProofAlphaN}
		\langle \iter{\alpha}{n}, \mu \rangle & \geq \veps \cdot (n-1) \cdot (1-\iter{q}{n})\,.
	\end{align}
	
	From Proposition \ref{prop:SinkhornAsymmetricMonotonicity} we know that there is some $y^\ast \in Y$ with $\iterl{v}(y^\ast)=1 \leq \iterll{\hat{v}}(y^\ast)$ for all iterates $\ell \geq 0$. So
	\begin{align*}
		1 \leq \iterll{\hat{v}}(y^\ast) = \frac{\nu(y^\ast)}{[\kernel^\T \iterll{u}](y^\ast)}
		= \frac{1}{\sum_{x \in X} \exp\big(-\tfrac{1}{\veps} [c(x,y^\ast)-\iterll{\alpha}(x)] \big)\,\mu(x)},
	\end{align*}
	from which we infer $\exp(-\cMax/\veps) \cdot \la \exp(\iterll{\alpha}/\veps), \mu \ra \leq 1$, i.e.\ $\la \iterll{u},\mu \ra \leq \exp(\cMax/\veps)$. With Jensen's inequality we eventually find $\la \iterll{\alpha}, \mu \ra \leq \cMax$ for $\ell \geq 0$.
	
	Combining this with \eqref{eqn:SinkhornAsymmetricIterationsBoundProofAlphaN} we obtain $n \leq 1 + \tfrac{\cMax}{\veps\,(1-\iter{q}{n})}$. So, as long as $\iter{q}{n} < \qTarget$ we have $n < 1 + \tfrac{\cMax}{\veps\,(1-\qTarget)}$. By contraposition we know that there is some $n \leq 2 + \tfrac{\cMax}{\veps\,(1-\qTarget)}$ such that $\iter{q}{n} \geq \qTarget$.
\end{proof}
And finally, we formally establish convergence of the iterates.
\begin{corollary}[Convergence of Asymmetric Algorithm]
	\label{cor:SinkhornAsymmetricConvergence}
	Ignoring the stopping criterion, the iterates $(\iterl{u},\iterl{v})$ of the asymmetric Algorithm\ \ref{alg:SinkhornAsymmetric} converge to a solution of the scaling problem and $\iterl{q} \rightarrow 1$.
\end{corollary}
\begin{proof}
	With the upper bound $\la \iterl{u}, \mu \ra \leq \exp(\cMax/\veps)$ (Proposition \ref{prop:SinkhornAsymmetricIterationBound}) we obtain the pointwise lower bound $\iterl{v}(y) \geq \exp(-\cMax/\veps)$ for all $\ell \geq 0$.
	Since $\iterl{v}$ is pointwise decreasing, it converges to some limit $\iter{v}{\infty} \geq \exp(-\cMax/\veps)>0$.
	
	The map $f : \iterl{v} \mapsto \iterll{v}$ is continuous for $\iterl{v} > 0$. With $\iterl{v} \rightarrow \iter{v}{\infty}$ and $\iterll{v} = f(\iterl{v}) \rightarrow \iter{v}{\infty}$ have $f(\iter{v}{\infty}) = \iter{v}{\infty}$ which implies that $\iter{v}{\infty}$ (together with the corresponding $\iter{u}{\infty} = \mu \oslash (\kernel\,\iter{v}{\infty})$) solves the scaling problem.
	This implies convergence of $\iterl{q}$ to $1$.
\end{proof}
\begin{remark}[On the Stopping Criterion and Relation to \cite{FranklinLorenz-Scaling-1989}]
\label{rem:SinkhornAsymmetricStoppingCriterion}
The criterion $q \geq \qTarget$ is motivated by Lemma \ref{lem:SinkhornAsymmetricAlphaIncrement}, to provide a minimal increment of $\alpha$ during iterations. $1-q$ measures the mass that is still missing and is equal to the $L^1$ error between the marginals of $\pi$ and the desired marginals $\mu$ and $\nu$.
In pathological cases the dual variables $(\alpha,\beta)$ may still be far from optimizers, even though $q \geq \qTarget$ (see Example \ref{exp:QNonConvergence}).
In \cite[Lemma 2]{FranklinLorenz-Scaling-1989} linear convergence of the marginals in the Hilbert projective metric is proven. This is a stricter measure of convergence, less prone to `premature' termination. However, for small $\veps$ the contraction factor is roughly $1-4\,\exp(-\cMax/\veps)$, which is impractical.
The scaling $\mc{O}(1/\veps)$ predicted by Proposition \ref{prop:SinkhornAsymmetricIterationBound} is consistent with numerical observations when one uses the $L^1$ or $L^\infty$ marginal error as stopping criterion (Sect.~\ref{sec:NumericsEfficiency}).
Therefore we consider the $q$-criterion to be a reasonable measure for convergence, as long as one keeps $1-\qTarget \ll \delta$ (Example \ref{exp:QNonConvergence}).
\end{remark}

\begin{example}
\label{exp:QNonConvergence}
We consider the $1 \times 2$ toy problem with the following parameters:
\begin{align*}
	\mu & =\begin{pmatrix} 1 \end{pmatrix}^\T, &
	\nu & =\begin{pmatrix} 1-\delta & \delta \end{pmatrix}^\T, &
	c & =\begin{pmatrix} 0 & C \end{pmatrix}, &
	\kernel & =\begin{pmatrix} 1-\delta & \delta  \cdot e^{-C/\veps} \end{pmatrix}
\end{align*}
for some $C>0$, $\delta \in (0,1)$ and some regularization strength $\veps>0$.
And we consider the scaling factors (one for $X$, two choices for $Y$):
$u = \begin{pmatrix} 1 \end{pmatrix}^\T$,
$v_1 = \begin{pmatrix} 1 & 1 \end{pmatrix}^\T$,
$v_2 = \begin{pmatrix} 1 & e^{C/\veps} \end{pmatrix}^\T$.
Let $\pi_i=\diag(u)\,\kernel\,\diag(v_i)$ and corresponding total masses $q_i$, $i=1,2$. We find:
\begin{align*}
	\pi_1 & =\begin{pmatrix} 1-\delta & \delta \cdot e^{-C/\veps} \end{pmatrix}, &
	q_1 & = 1-\delta\,(1-e^{-C/\veps}), &
	\pi_2 & =\begin{pmatrix} 1-\delta & \delta \end{pmatrix}, &
	q_2 & = 1.
\end{align*}
$\pi_2$ and $(\alpha,\beta_2) = \veps\,\log(u,v_2)$ are primal and dual solutions. $\pi_1$ is sub-feasible (see Proposition \ref{prop:SinkhornAsymmetricMonotonicity}). For fixed $\veps>0$, as $\delta \to 0$, $q_1$ tends to 1 (but is strictly smaller), i.e.~the pair $(u,v_1)$ has almost converged in the $q$-measure sense, but the distance between $\beta_1 = \veps\,\log v_1$ and the actual solution $\beta_2$ is $C$.
\end{example}

\begin{remark}[Analogy to Auction Algorithm]
	\label{rem:AnalogyAuction}
	For now assume $|X|=|Y|=N$ and $\mu$, $\nu$ are normalized counting measures. Then line \ref{alg:SinkhornAsymmetric:uvHat} in Algorithm \ref{alg:SinkhornAsymmetric}, expressed in dual variables, becomes
	\begin{align*}
		\alpha(x) & \gets \softmin(\{c(x,y)-\beta(y) | y \in Y\},\veps) + \veps\,\log N\,, \\
		\hat{\beta}(y) & \gets \softmin(\{c(x,y)-\alpha(x) | x \in X\},\veps) + \veps \, \log N\,.
	\end{align*}
	These are formally similar to the corresponding lines \ref{alg:Auction:Alpha} and \ref{alg:Auction:Beta} in Algorithm \ref{alg:Auction}.
	We can interpret the $u$-update in Algorithm \ref{alg:SinkhornAsymmetric} as $x$ not just submitting a bid to the best candidate $y$, but to all candidates, weighted by the attractiveness (recall that in the Sinkhorn algorithm, a change in the dual variable directly implies a change in the primal iterate via \eqref{eqn:EntropyPDRelation}).
	Conversely, in line \ref{alg:SinkhornAsymmetric:vMin}, $y$ does not only accept the best bid, but bids from all candidates, again weighted by price. If there are too many bids (i.e.~if $\hat{v}(y) < v(y)$), $\beta(y)$ decreases and thereby rejects superfluous offers.
	
	Consequently, in Algorithm \ref{alg:SinkhornAsymmetric} one can observe that points in $X$ compete for the mass in $Y$ in a way similar to the auction algorithm by repeatedly increasing their prices until a different target seems more attractive or other competitors lose interest. In both algorithms the minimal increment is related to the parameter $\veps$ which leads to iteration bounds that are proportional to $1/\veps$ (Props.~\ref{prop:AuctionMonotonicity} and \ref{prop:SinkhornAsymmetricIterationBound}).
	An attempt to mimic the analysis of $\veps$-scaling is made in Section \ref{sec:AuctionEpsScaling} (cf.~Remark \ref{rem:AuctionSinkhornEpsScaling}).

	One can interpret the standard Sinkhorn algorithm with $v \gets \hat{v}$ as $y$ submitting a `counter-bid' if it has not received enough bids.
	Such a symmetrization has also been discussed for the auction algorithm. But then the complexity analysis based on monotonous dual variables breaks down, and the algorithm may even run indefinitely (see `down iterations' in \cite{Bertsekas-MinCostFlowReview1988}).
\end{remark}

\subsection{Stability of Dual Solutions}
\label{sec:AuctionStability}

The main result of this section is Theorem \ref{thm:AuctionEpsStability}, which provides stability of dual solutions to entropy regularized optimal transport (Def.~\ref{def:EntropicOptimalTransport}) under changes of the regularization parameter $\veps$.
Its implications for $\veps$-scaling are discussed in Sect.~\ref{sec:AuctionEpsScaling}.

We consider a similar setup as in Sect.~\ref{sec:AuctionSimple}: $\mu \in \prob(X)$, $\nu \in \prob(Y)$, $c \in \R_+^{X \times Y}$. Again, the reference measure $\rho$ for regularization, see \eqref{eqn:Kernel}, is chosen to be the product measure $\rho(x,y) = \mu(x) \cdot \nu(y)$.
For Theorem \ref{thm:AuctionEpsStability} we introduce an additional assumption on $\mu$ and $\nu$. The necessity of this assumption can be demonstrated by counter-examples similar to Example \ref{exp:QNonConvergence}.
\begin{assumption}[Atomic Mass]
	\label{asp:AuctionAtomicMass}
	For $\mu \in \prob(X)$, $\nu \in \prob(Y)$ there is some $M \in \N$ such that $\mu = r/M$, $\nu = s/M$ for $r \in \N^{X}$, $s \in \N^{Y}$.
\end{assumption}

\begin{theorem}[Stability of Dual Solutions under $\veps$-Scaling]
	\label{thm:AuctionEpsStability}
	Let $\max\{|X|,|Y|\} \leq N < \infty$ and let $\mu$ and $\nu$ satisfy Assumption \ref{asp:AuctionAtomicMass} for some $M \in \N$.
	For two regularization parameters $\veps_1 > \veps_2 > 0$, let $(\alpha_1,\beta_1)$ and $(\alpha_2,\beta_2)$ be maximizers of the corresponding dual regularized optimal transport problems (Def.~\ref{def:EntropicOptimalTransport}) and let $\Delta \alpha = \alpha_2-\alpha_1$ and $\Delta \beta = \beta_2 - \beta_1$. Then
	\begin{subequations}
	\label{eqn:AuctionEpsStability}
	\begin{align}
		\max \Delta \alpha - \min \Delta \alpha & \leq \veps_1 \cdot N \cdot (4 \log N + 24 \log M), \\
		\max \Delta \beta - \min \Delta \beta & \leq \veps_1 \cdot N \cdot (4 \log N + 24 \log M).
	\end{align}
	\end{subequations}
\end{theorem}
\begin{remark}[Relation to \cite{Cominetti-ExpBarrierConvergence-1992} and Motivation]
	\label{rem:AuctionStabilityMotivation}
	\cite{Cominetti-ExpBarrierConvergence-1992} studies the convergence of entropy regularized linear programs to the unregularized variant and can be used to understand the limit of entropy regularized optimal transport (Def.~\ref{def:EntropicOptimalTransport}).
	To apply \cite{Cominetti-ExpBarrierConvergence-1992}, the constraint matrix must have full rank and the set of optimal solutions to \eqref{eqn:OptimalTransportDual} must be bounded. When the cost $c$ is finite this is achieved by arbitrarily fixing one dual variable, e.g.~$\alpha(x_0)=0$, and removing the corresponding column from the dual constraint matrix.
	The slight difference in the definition of the entropy (or the dual exponential barrier) can be absorbed into a change of variables which converges to the identity in the limit $\veps \to 0$.
	
	Then \cite[Props.~3.1 and 3.2]{Cominetti-ExpBarrierConvergence-1992} imply that the optimal solutions of \eqref{eqn:EntropicOTDual} remain bounded and converge to a particular solution of \eqref{eqn:OptimalTransportDual} as $\veps \to 0$.
	Furthermore, \cite{Cominetti-ExpBarrierConvergence-1992} provides statements about the convergence of the optimal couplings (Prop.~4.1) and the asymptotic behaviour (Thm.~5.8).
	
	The bounds derived in \cite{Cominetti-ExpBarrierConvergence-1992} depend on the geometry of the primal and dual feasible polytopes of \eqref{eqn:OptimalTransport}, i.e.~on the transport cost function $c$. In contrast, the bound of Thm.~\ref{thm:AuctionEpsStability} does not depend on $c$.
	The motivation for deriving such a bound is the implication for $\veps$-scaling, see Section \ref{sec:AuctionEpsScaling}. Note that Thm.~\ref{thm:AuctionEpsStability} also implies that the optimal dual variables remain bounded as $\veps \to 0$.
\end{remark}

\begin{remark}[Proof Strategy]
	The proof requires several auxiliary definitions and lemmas. The estimate consists of two contributions: One stems from following paths within connected components of what we call \emph{assignment graph} (defined in the following lemma), using the primal-dual relation \eqref{eqn:EntropyPDRelation}. This reasoning is analogous to the proof strategy for $\veps$-scaling in the auction algorithm (see \cite{Bertsekas-MinCostFlowReview1988}).
	However, between different connected components \eqref{eqn:EntropyPDRelation} is too weak to yield useful estimates. So a second contribution arises from a stability analysis of \emph{effective diagonal problems} (in Lemmas \ref{lem:EffectiveDiagonalProblem} and \ref{lem:EffectiveDiagonalProblemStability}).
\end{remark}
\begin{lemma}[Assignment Graph]
	\label{lem:AssignmentGraph}
	For two feasible couplings $\pi_1$, $\pi_2 \in \Pi(\mu,\nu)$ and a threshold $M^{-1} \geq 1$ the corresponding assignment graph is a bipartite directed graph with vertex sets $(X,Y)$ and the set of directed edges
	\begin{align*}
		\edges & = \{ (x,y) \in X \times Y  : \pi_2(x,y) \geq \mu(x) \cdot \nu(y) / M \} \sqcup \{ (y,x) \in Y \times X  : \pi_1(x,y) \geq \mu(x) \cdot \nu(y) / M \}
	\end{align*}
	where $(a,b) \in \edges$ indicates a directed edge from $a \to b$.

	The assignment graph has the following properties:
	\begin{enumerate}[(i)]
		\item Every node has at least one incoming and one outgoing edge.
			\label{item:AssignGraphOneEdge}
		\item Let $X_0 \subset X$, $Y_0 \subset Y$ such that there are no outgoing edges from $(X_0,Y_0)$ to the rest of the vertices, then $|\mu(X_0) - \nu(Y_0)| < 1/M$. This is also true when there are no incoming edges from the rest of the vertices.
		If $\mu$ and $\nu$ are atomic, with atom size $1/M$ (see Assumption \ref{asp:AuctionAtomicMass}), then $\mu(X_0) = \nu(Y_0)$.
			\label{item:AssignGraphBalanced}
		\item Let $\mu$ and $\nu$ be atomic, with atom size $1/M$. Let $\{(X_i,Y_i)\}_{i=1}^\nK$ be the vertex sets of the strongly connected components of the assignment graph, for some $\nK \in \N$ (taking into account the orientation of the edges).
		Then the sets $\{X_i\}_{i=1}^\nK$ and $\{Y_i\}_{i=1}^\nK$ are partitions of $X$ and $Y$, and $\mu(X_i) = \nu(Y_i)$ for $i=1,\ldots,\nK$.
		\label{item:AssignGraphCycleBalanced}
	\end{enumerate}
\end{lemma}
\begin{proof}
	Assume, a node $x \in X$ had no outgoing edge. Then $\sum_{y \in Y} \pi_2(x,y) < \mu(x) / M \leq \mu(x)$. This contradicts $\pi_2 \in \Pi(\mu,\nu)$. Existence of incoming edges follows analogously.
	
	Let $\hat{X}_0 = X \setminus X_0$, $\hat{Y}_0 = Y \setminus Y_0$. If $(X_0,Y_0)$ has no outgoing edges, then
	\begin{align*}
		\sum_{\substack{x \in X_0,\\y \in \hat{Y}_0}} \pi_2(x,y) < 
			\sum_{\substack{x \in X_0,\\y \in \hat{Y}_0}} \mu(x) \cdot \nu(y) / M
			\leq \tfrac{1}{M},
		\quad
		\sum_{\substack{x \in \hat{X}_0,\\y \in Y_0}} \pi_1(x,y) & < 
		\sum_{\substack{x \in \hat{X}_0,\\y \in Y_0}} \mu(x) \cdot \nu(y) / M \leq \tfrac{1}{M}\,.
	\end{align*}
	Since $\pi_1$, $\pi_2 \in \Pi(\mu,\nu)$, the first inequality implies $\mu(X_0) = \pi_2(X_0 \times Y) = \pi_2(X_0 \times Y_0) + \pi_2(X_0 \times \hat{Y}_0) < \nu(Y_0)+ 1/M$ and the second inequality implies $\nu(Y_0) < \mu(X_0) + 1/M$, i.e.\ $|\mu(X_0) - \nu(Y_0)| < 1/M$. With Assumption \ref{asp:AuctionAtomicMass} for atom size $1/M$, this implies $\mu(X_0) = \nu(Y_0)$. The statement about incoming edges follows from $\mu(\hat{X}_0) = 1-\mu(X_0)$ and $\nu(\hat{Y}_0) = 1-\nu(Y_0)$.

	Every node in $(X,Y)$ is part of at least one strongly connected component (containing at least the node itself). If two strongly connected components have a common element, they are identical. Hence, the strongly connected components form partitions of $X$ and $Y$.
	For some $x \in X$ (or $y \in Y$), let $X_{\tn{out}} \subset X$ and $Y_{\tn{out}} \subset Y$ be the set of nodes that can be reached from $x$, let $X_{\tn{in}} \subset X$ and $Y_{\tn{in}} \subset Y$ be the set of nodes from which one can reach $x$ and let $(X_{\tn{con}} = X_{\tn{out}} \cap X_{\tn{in}},Y_{\tn{con}} = Y_{\tn{out}} \cap Y_{\tn{in}})$ be the strongly connected component of $x$.
	Clearly $(X_{\tn{out}},Y_{\tn{out}})$ has no outgoing edges.
	Hence, by (\ref{item:AssignGraphBalanced}) one has $\mu(X_{\tn{out}}) = \nu(Y_{\tn{out}})$.
	Moreover, $(X_{\tn{out}} \setminus X_{\tn{in}},Y_{\tn{out}} \setminus Y_{\tn{in}})$ has no outgoing edges, hence $\mu(X_{\tn{out}} \setminus X_{\tn{in}}) = \nu(Y_{\tn{out}} \setminus Y_{\tn{in}})$, from which follows that $\mu(X_{\tn{con}}) = \nu(Y_{\tn{con}})$.
\end{proof}

\begin{lemma}[Reduction to Effective Diagonal Problem]
	\label{lem:EffectiveDiagonalProblem}
	Let $\{X_i\}_{i=1}^\nK$ and $\{Y_i\}_{i=1}^\nK$ be partitions of $X$ and $Y$, for some $\nK \in \N$, with $\mu(X_i) = \nu(Y_i)$ for $i = 1,\ldots,\nK$. Let $\{y_i\}_{i=1}^\nK \subset Y$ such that $y_i \in Y_i$.
	Let $(\alpha^\dagger,\beta^\dagger)$ be optimizers for the dual entropy regularized optimal transport problem (Def.~\ref{def:EntropicOptimalTransport}) for a regularization parameter $\veps>0$.
	
	Consider the following functional over $\R^\nK$:
	\begin{align*}
		\Jdiag : \R^\nK & \to \R, &
		\hat{\beta} \mapsto -\veps \sum_{i,j=1}^\nK \exp\left(-\tfrac{1}{\veps}
			\left[
				d(i,j)+\hat{\beta}(i) - \hat{\beta}(j)
			\right]
			\right)
	\end{align*}
	where $d \in \R^{\nK \times \nK}$ with
	\begin{align}
		\label{eqn:EffectiveDiagonalProblemCoefs}
		d(i,j) & = - \veps \log\left(
			\sum_{\substack{x \in X_i\\y \in Y_j}}
			\exp\left(-\tfrac{1}{\veps}\left[
				c(x,y)-\alpha^\dagger(x) - \beta^\dagger(y_i) - \beta^\dagger(y) + \beta^\dagger(y_j)
						\right]
					\right) \cdot \mu(x) \cdot \nu(y)
			\right)\,.
	\end{align}
	Then $\hat{\beta}^\dagger \in \R^\nK$, given by $\hat{\beta}^\dagger(i) = \beta^\dagger(y_i)$, is a maximizer of $\Jdiag$. Conversely, if $\hat{\beta}^{\dagger \dagger}$ is a maximizer of $\Jdiag$, then there is a constant $b \in \R$, such that $\hat{\beta}^{\dagger \dagger}(i) = \hat{\beta}^{\dagger}(i) + b$ for all $i \in 1,\ldots,\nK$. 
\end{lemma}
\begin{proof}
	We define the functional $\Jdiag : \R^{\nK} \to \R$ as follows:
	\begin{align*}
		\Jdiag : \hat{\beta} \mapsto \Jdual\left(
			\begin{pmatrix}
				\tilde{\alpha} \\
				\tilde{\beta}
			\end{pmatrix}
			+
			\begin{pmatrix}
				-B_X \\ B_Y
			\end{pmatrix}
			\hat{\beta}
			\right)
	\end{align*}
	where $\Jdual$ denotes the dual functional of entropy regularized optimal transport \eqref{eqn:EntropicOTDual}, and
	\begin{itemize}
		\item $\tilde{\alpha} \in \R^{X}$ with $\tilde{\alpha}(x) = \alpha^{\dagger}(x)+\beta^{\dagger}(y_i)$ when $x \in X_i$;
		\item $\tilde{\beta} \in \R^{Y}$ with $\tilde{\beta}(y) = \beta^{\dagger}(y) - \beta^{\dagger}(y_j)$ when $y \in Y_j$;
		\item $B_X \in \R^{X \times \nK}$ with $B_X(x,i) = 1$ if $x \in X_i$ and $0$ else;
		\item $B_Y \in \R^{Y \times \nK}$ with $B_Y(y,j) = 1$ if $y \in Y_j$ and $0$ else.
	\end{itemize}
	Then one has
	\begin{align*}
		\begin{pmatrix}
			\alpha^{\dagger} \\ \beta^{\dagger}
		\end{pmatrix}
		=
			\begin{pmatrix}
				\tilde{\alpha} \\
				\tilde{\beta}
			\end{pmatrix}
			+
			\begin{pmatrix}
				-B_X \\ B_Y
			\end{pmatrix}
			\hat{\beta}^{\dagger}
			\,.
	\end{align*}
	Since maximizing $\Jdiag$ corresponds to maximizing $\Jdual$ over an affine subspace, clearly $\hat{\beta}^{\dagger}$ is a maximizer of $\Jdiag$.
	Since $\Jdiag$ inherits the invariance of $\Jdual$ under constant shifts, any $\hat{\beta}^{\dagger \dagger}$ of the form given above, is also a maximizer.
	Consequently, we may add the constraint $\hat{\beta}(1)=0$, which does not change the optimal value. With this added constraint the functional becomes strictly convex, which implies a unique optimizer. Hence, any optimizer of the unconstrained functional can be written in the form of $\hat{\beta}^{\dagger \dagger}$.
	
	Let us now give a more explicit expression of $\Jdiag(\hat{\beta})$. We find
	\begin{align*}
		\Jdiag(\hat{\beta}) & = \la B_Y^\top \nu - B_X^\top \mu, \hat{\beta} \ra
			-\veps \sum_{i,j=1}^\nK
				\sum_{\substack{x \in X_i\\y \in Y_j}}
					\exp\left(
						-\tfrac{1}{\veps}\left[
							c(x,y)-\tilde{\alpha}(x) + \hat{\beta}(i) - \tilde{\beta}(y) - \hat{\beta}(j)
						\right]
					\right) \cdot \mu(x) \cdot \nu(y) \\
			& \qquad + \la \mu,\tilde{\alpha} \ra + \la \nu, \tilde{\beta} \ra  + \veps \cdot \kernel(X \times Y)\,.
	\end{align*}
	Note that the second line is constant w.r.t.~$\hat{\beta}$. Since $\mu(X_i) = \nu(Y_i)$ the linear term vanishes and we can write $\Jdiag(\hat{\beta}) = -\veps \sum_{i,j=1}^\nK \exp\left(-\tfrac{1}{\veps} \left[d(i,j) + \hat{\beta}(i) - \hat{\beta}(j) \right] \right) + \const$ with coefficients $d \in \R^{\nK \times \nK}$, as given above. The constant offset does not affect minimization.
\end{proof}

\begin{lemma}[Effective Diagonal Problem and Stability]
	\label{lem:EffectiveDiagonalProblemStability}
	For a parameter $\veps > 0$ and a real matrix $d \in \R^{\nK \times \nK}$ consider the following functional:
	\begin{align}
		\Jdiag_{\veps,d}(\beta) & = \sum_{i,j=1}^\nK \exp\left(\left[-d(i,j)-\beta(i) + \beta(j) \right]/\veps
			\right)
	\end{align}
	Minimizers of $\Jdiag_{\veps,d}$ exist.

	Let $\veps_1 \geq \veps_2>0$ be two parameters and $d_1$, $d_2 \in \R^{\nK \times \nK}$ two real matrices. Let $\beta_1^\dagger$ and $\beta_2^\dagger$ be minimizers of $\Jdiag_{\veps_1,d_1}$ and $\Jdiag_{\veps_2,d_2}$, let $\Delta d = d_2 - d_1$, $\Delta \beta = \beta_2^\dagger - \beta_1^\dagger$. Let the matrix $w \in \R^{\nK \times \nK}$ be given by $w(i,j) = \max\{-\Delta d(i,j),\Delta d(j,i)\}$.
	Then $\max \Delta \beta - \min \Delta \beta \leq \maxdiam(w) + 2\,\veps_1\,\nK\,\log \nK$, where
	\begin{align*}
		\maxdiam(w) = \max \left\{ \sum_{i=1}^{k-1} w(j_i,j_{i+1})\,:\,
			k \in \{2,\ldots,\nK\},
			j_i \in \{1,\ldots,\nK\} \tn{ for } i=1,\ldots,k, \tn{ all } j_i \tn{ distinct.}
			\right\}.
	\end{align*}
	That is, $\maxdiam(w)$ is the length of the longest cycle-less path in $\{1,\ldots,\nK\}$ with edge lengths $w$.
\end{lemma}

The proofs of Theorem \ref{thm:AuctionEpsStability} and Lemma \ref{lem:EffectiveDiagonalProblemStability} can be found in Appendix \ref{sec:ApxAdditionalProofs}.

\subsection[Application to Epsilon-Scaling]{Application To $\veps$-Scaling}
\label{sec:AuctionEpsScaling}

Assuming that we know the dual solution for some $\veps_1>0$, then Theorem \ref{thm:AuctionEpsStability} allows to bound the number of iterations of Algorithm \ref{alg:SinkhornAsymmetric} for some smaller $\veps_2 \in (0,\veps_1)$, independently of bounds on the cost function $c$.
This may have implications for the efficiency of $\veps$-scaling (see Remark \ref{rem:AuctionSinkhornEpsScaling}).
\begin{proposition}[Single $\veps$-Scaling Step]
\label{prop:EpsScaling}
Consider the set-up of Theorem \ref{thm:AuctionEpsStability}. In particular, let $\veps_1 > \veps_2 > 0$ be two regularization parameters, let $(\alpha_1,\beta_1)$, $(\alpha_2,\beta_2)$ be corresponding optimizers of \eqref{eqn:EntropicOTDual}.
If Algorithm \ref{alg:SinkhornAsymmetric} is initialized with $\iterz{v}=\exp(\beta_1/\veps_2)$, with regularization $\veps_2$, and for a given $\qTarget \in (0,1)$, the number of iterations $n$ necessary to achieve $\iter{q}{n} \geq \qTarget$ is bounded by
	\begin{align}
		\label{eqn:EpsScaling}
		n \leq 2 + \frac{\veps_1}{\veps_2} \frac{N \cdot (4 \log N + 24 \log M) + \log M}{1-\qTarget}\,.
	\end{align}
\end{proposition}

\begin{proof}
	For the optimal scaling factor $u_1$ of the $\veps_1$-problem we find:
	\begin{align*}
	u_1(x) & \eqdef \exp(\alpha_1(x)/\veps_1) = \Big( \sum_{y \in Y}
		\exp\left(-\tfrac{1}{\veps_1} [c(x,y)-\beta_1(y)]\right) \nu(y) \Big)^{-1}
	\end{align*}
	This implies $u_1(x)^{-1}\,\nu(y)^{-1} \geq \exp\left(-\tfrac{1}{\veps_1} [c(x,y)-\beta_1(y)]\right)$ for all $(x,y) \in X \times Y$.
	With this we can bound the first iterate of the $\veps_2$-run of the algorithm by:
	\begin{align*}
	\itero{u}(x) = \Big( \sum_{y \in Y}
		\exp\left(-\tfrac{1}{\veps_2} [c(x,y)-\beta_1(y)]\right) \nu(y) \Big)^{-1}
		\geq \Big( \sum_{y \in Y}
			\left(u_1(x)\,\nu(y)\right)^{-\tfrac{\veps_1}{\veps_2}} \,\nu(y) \Big)^{-1}
		\geq \Big(\tfrac{u_1(x)}{M}\Big)^{\tfrac{\veps_1}{\veps_2}}
	\end{align*}
	where we have used $\nu(y)\geq 1/M$, Assumption \ref{asp:AuctionAtomicMass}. Eventually we find
	$\itero{\alpha}(x) \geq \alpha_1(x) -\veps_1\,\log M$.

	By monotonicity of the iterates we have $\beta_2 \leq \iterl{\beta} \leq \iterz{\beta} = \beta_1$ and $\iterl{\beta}(y')=\beta_1(y')$ for a suitable $y' \in Y$ (Proposition \ref{prop:AuctionMonotonicity}). Consequently $\max \Delta \beta=0$.
	Then, from Theorem \ref{thm:AuctionEpsStability}, we obtain $\beta_2(y) - \beta_1(y) \geq \min \Delta \beta \geq -\veps_1 \cdot A$ where $A=N \cdot (4 \log N + 24 \log M)$.
	With this we can bound the $u$-iterates:
	\begin{align*}
		\iterl{u}(x) & \leq u_2(x) \eqdef \exp(\alpha_2(x)/\veps_2) =
			\Big( \sum_{y \in Y} \exp\left(- \tfrac{1}{\veps_2}
				[c(x,y)-\beta_2(y)]\right) \nu(y) \Big)^{-1} \\
		& \leq \Big( \sum_{y \in Y} \exp\left(- \tfrac{1}{\veps_2}
				[c(x,y)-\beta_1(y)]\right) \nu(y) \Big)^{-1} \,
					\exp\left(\tfrac{\veps_1}{\veps_2} A\right) \\
		\intertext{With convexity of $s \mapsto s^{\veps_1/\veps_2}$ and Jensen's inequality we get}
		\iterl{u}(x) & \leq \Big( \sum_{y \in Y} \exp\left(- \tfrac{1}{\veps_1}
				[c(x,y)-\beta_1(y)]\right) \nu(y) \Big)^{-\veps_1/\veps_2} \,
					\exp\left(\tfrac{\veps_1}{\veps_2} A\right)
			= u_1(x)^{\veps_1/\veps_2} \exp\left(\tfrac{\veps_1}{\veps_2} A\right)
	\end{align*}
	and finally $\iterl{\alpha}(x) \leq \alpha_1(x) + \veps_1\,A$. We summarize: $\iterl{\alpha}(x)-\iterz{\alpha}(x) \leq \veps_1\,(A+\log M)$.
	Now using Lemma \ref{lem:SinkhornAsymmetricAlphaIncrement} and arguing as in Proposition \ref{prop:SinkhornAsymmetricIterationBound}, we find that there is some $n \leq 2 + \frac{\veps_1}{\veps_2} \frac{A + \log M}{1-\qTarget}$ such that $\iter{q}{n} \geq \qTarget$.
\end{proof}

Let now $\cMax = \max c$ for a cost function $c \geq 0$, let $\hat{\veps}>0$ be the desired final regularization parameter, pick some $\lambda \in (0,1)$ and let $k \in \N$ such that $\hat{\veps} \cdot \lambda^{-k} \geq \cMax$. Let $\epsList=(\hat{\veps} \cdot \lambda^{-k}, \hat{\veps} \cdot \lambda^{-k+1}, \ldots, \hat{\veps})$ be a list of decreasing regularization parameters.

\begin{remark}
\label{rem:AuctionSinkhornEpsScaling}
Now we combine Algorithm \ref{alg:SinkhornAsymmetric} with $\veps$-scaling, (cf.~Algorithm \ref{alg:ScalingEpsScaling}).
For $\veps=\hat{\veps} \cdot \lambda^{-k} \geq \cMax$, according to Proposition \ref{prop:SinkhornAsymmetricIterationBound} it will take at most $2+\frac{1}{1-\qTarget}$ iterations.
It is tempting to deduce from Proposition \ref{prop:EpsScaling} that for each subsequent value of $\veps$ at most $2+\frac{A}{\lambda\,(1-\qTarget)}$ iterations are required, with $A=N \cdot (4 \log N + 24 \log M) + \log M$. For $N>1$ the total number of iterations would then be bounded by $(2+\frac{A}{\lambda\,(1-\qTarget)}) \cdot (k+1)$.
For fixed $\lambda$ the step parameter $k$ scales like $\log(\cMax/\hat{\veps})$. Consequently, the total number of iterations would be bounded by $\mc{O}(\log(\cMax/\hat{\veps}))$ w.r.t.~the cost function and regularization, which would be analogous to $\veps$-scaling for the auction algorithm (Remark \ref{rem:AuctionEpsScaling}).

There is an obvious gap in this reasoning: Theorem \ref{thm:AuctionEpsStability} assumes that $\beta_1$ is known exactly, while Algorithm \ref{alg:SinkhornAsymmetric} only provides an approximate result. From Example \ref{exp:QNonConvergence} we learn that in extreme cases this difference can be substantial and disrupt the efficiency of $\veps$-scaling.
Thus, additional assumptions on the problem are required to make the above argument rigorous.

However, as discussed in Remark \ref{rem:SinkhornAsymmetricStoppingCriterion}, in practice we usually observe that approximate iterates are sufficient and we can therefore hope that $\veps$-scaling does indeed serve its purpose.
\end{remark}


\newcommand{\piVec}{\bm{\pi}}

\section{Numerical Examples}
\label{sec:Numerics}
Now we present a series of numerical experiments to confirm the usefulness of the modifications proposed in Sect.~\ref{sec:Algorithm}.
We show that runtime and memory usage are reduced substantially.
At the same time the adapted algorithm is still as versatile as the basic version of \cite{ChizatEntropicNumeric2018}, Algorithm \ref{alg:Scaling}. But Algorithm \ref{alg:Full} can solve larger problems at lower regularization, yielding very sharp results.
We give examples for unbalanced transport, barycenters and Wasserstein gradient flows.
The code used for the numerical experiments is available from the author's website.\footnote{\url{https://github.com/bernhard-schmitzer}}

\subsection{Setup}
We transport measures on $[0,1]^d$ for $d \in \{1,2,3\}$, represented by discrete equidistant Cartesian grids.
The distance between neighbouring grid points is denoted by $h$.
For the squared Euclidean distance cost function $c(x,y)=|x-y|^2$, $x$, $y \in \R^d$, $\kernel$ is a Gaussian kernel with approximate width $\sqrt{\veps}$.
Therefore, it is useful to measure $\veps$ in units of $h^2$. For $\veps = h^2$ the blur induced by the entropy smoothing is on the length scale of one pixel.
With the enhanced scaling algorithm we solve most problems in this section with $\veps=0.1 \cdot h^2$, leaving very little blur and giving a good approximation of the original unregularized problem (see Fig.~\ref{fig:expEpsComparison}).

Unless stated otherwise, we use the following settings:
Test measures are mixtures of Gaussians, with randomized means and variances. The cost function is the squared Euclidean distance.
$\rho$ is the product measure $\mu \otimes \nu$ for optimal transport problems and the discretized Lebesgue measure on the product space for problems with variable marginals.
For standard optimal transport the stopping criterion is the $L^\infty$ error between prescribed marginals $(\mu,\nu)$ and marginals of the primal iterate $\pi$ (and likewise for Wasserstein barycenters). For all other models the primal-dual gap is used.
We set $\theta=10^{-20}$ for truncating the kernel and $\tau=10^2$ as upper bound for $(\uRel,\vRel)$ (cf.~\eqref{eqn:getKernelSparse}, Algorithm \ref{alg:ScalingStabilized}, line \ref{alg:ScalingStabilizedTauCheck}), implying a bound of $10^{-16} \cdot \rho(X \times Y)$ for the truncation error, which is many orders of magnitude below prescribed marginal accuracies or primal-dual gaps.
The hierarchical partitions in the coarse-to-fine scheme are $2^d$-trees, where each layer $i$ is a coarser $d$-dimensional grid with grid constant $h_i$.
For combination with $\veps$-scaling (Algorithm \ref{alg:Full}) we choose the lists $\epsList_i$, $i>0$, such that for the smallest $\veps_i$ in each $\epsList_i$ we have roughly $\veps_i/h_i^2 \approx 1$. On the finest scale, we go down to the desired final value of $\veps$.
All reported run-times were obtained on a single core of an Intel Xeon E5-2697 processor at 2.7 GHz.

\subsection{Efficiency of Enhanced Algorithm}
\label{sec:NumericsEfficiency}
\begin{figure}[hbt]
	\centering
	\includegraphics{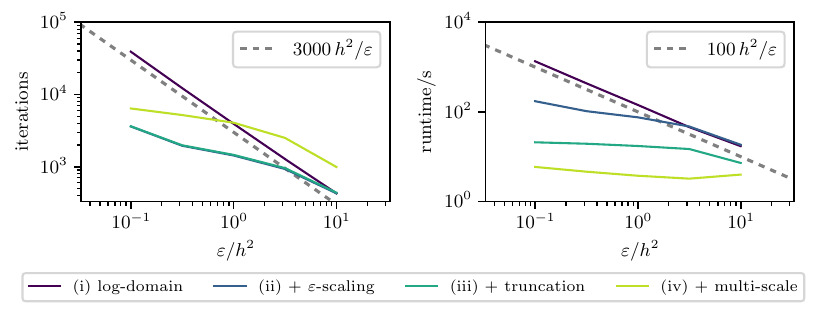}
	\caption{%
	Efficiency of enhancements: average number of iterations and runtime for different $\veps$ and algorithms. %
	$X=Y$ are 2-d $64 \times 64$ grids. %
	(i) log-domain stabilized, Algorithm \ref{alg:ScalingStabilized}, %
	(ii) with $\veps$-scaling, Algorithm \ref{alg:ScalingEpsScaling}, %
	(iii) with sparse stabilized kernel, \eqref{eqn:getKernelSparse}, %
	(iv) with multi-scale scheme, Algorithm \ref{alg:Full}. %
	(ii) and (iii) need same number of iterations, but the sparse kernel requires less time. %
	The naive implementation, Algorithm \ref{alg:Scaling}, requires same number of iterations as (i), but numerical overflow occurs at approximately $\veps \leq 	3\,h^2$. %
	}
	\label{fig:expEnhancements}
\end{figure}

The numerical efficiency of the subsequent modifications presented in Sect.~\ref{sec:Algorithm}, applied to the standard Sinkhorn algorithm, is illustrated in Fig.~\ref{fig:expEnhancements}.
While the stabilized algorithm (i) is not yet faster than the naive implementation, it can robustly solve the problem for all given values of $\veps$.
The required number of iterations scales like $\mc{O}(1/\veps)$, in good agreement with the complexity analysis of Sect.~\ref{sec:AuctionSimple}.
With $\veps$-scaling (ii) the number of iterations is decreased substantially.
Replacing the dense kernel with the adaptive truncated sparse kernel (iii) does not change the number of required iterations, but saves time and memory.
With the multi-scale scheme the required number of iterations is slightly increased, since the initial dual variables obtained at a coarser level are only approximate solutions. But by reducing the number of variables during the early $\veps$-scaling stages, the runtime is further decreased (cf.~Fig.~\ref{fig:MultiScaleEpsScaling}).
	The combination of all modifications leads to an average total speed-up of more than two orders of magnitude on this problem type.

\begin{figure}[hbt]
	\centering%
	\hfill%
	\includegraphics{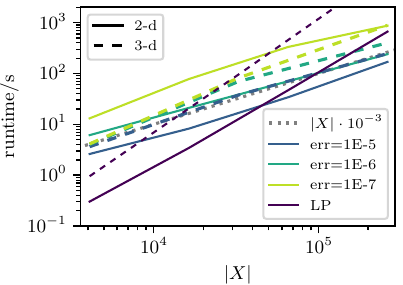}%
	\hfill%
	\includegraphics{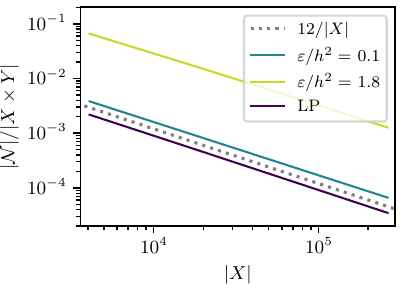}%
	\hfill%
	\caption{%
	Average runtime and sparsity of Algorithm \ref{alg:Full} for transporting test-images of different size (up to $512^2$ pixels for 2-d, $64^3$ for 3-d). Stopping criterion: $L^\infty$-marginal error, for different accuracy limits, final $\veps=0.1 \cdot h^2$. Performance of the adaptive sparse linear programming solver \cite{SchmitzerShortCuts2015} given for comparison (LP). %
	As expected, runtime increases with required accuracy. %
	The runtime of the scaling algorithm scales more favourably (approximately linear) with $|X|$ and is competitive for large instances. %
	The number of variables scales as $\mc{O}(1/|X|)$, suggesting that the number of variables per $x \in X$ is roughly constant.
	For the final $\veps=0.1 \cdot h^2$, the sparsity of the truncated kernel is comparable to \cite{SchmitzerShortCuts2015}. For $\veps=1.8 \cdot h^2$, the largest value in $\epsList_0$ (the list for the finest scale), more variables are required. %
	}
	\label{fig:expImageBenchmark}
\end{figure}

A runtime benchmark and study of the sparsity of the truncated kernel are given in Fig.~\ref{fig:expImageBenchmark}.
The runtime scales approximately linear with $|X|$ and for large problems the algorithm becomes faster than the adaptive sparse linear programming solver \cite{SchmitzerShortCuts2015}.
The final number of variables in the sparse kernel is comparable with the number of variables in \cite{SchmitzerShortCuts2015}, for higher values of $\veps$, during scaling, more memory is required (cf.~Fig.~\ref{fig:expEpsComparison}). This underlines again the importance of the coarse-to-fine scheme (Sect.~\ref{sec:AlgorithmMultiScale}).
It should be noted, that Fig.~\ref{fig:expEnhancements} shows results for $64 \times 64$ images, the smallest image size in Fig.~\ref{fig:expImageBenchmark}. For larger images the runtime difference between (i-iv) would be even larger, but due to time and memory constraints, only (iv) can be run practically.

\begin{figure}[hbt]
	\centering
	\includegraphics{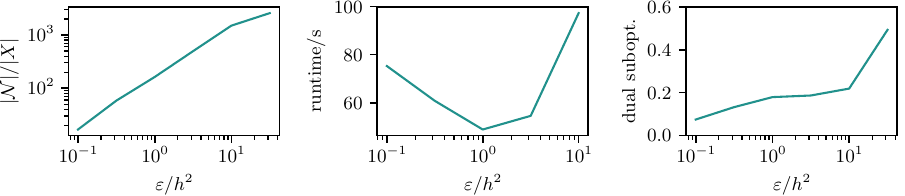}
	\caption{%
	Different final values for $\veps$ in Algorithm \ref{alg:Full}. $X=Y$ are 2-d $256 \times 256$ grids. %
	\textbf{Left} Average number of variables in truncated kernel per $x \in X$. For $\veps = 0.1 \cdot h^2$ only about 10 variables per $x \in X$ are required. As expected, this number increases with $\veps$ (cf.~Fig.~\ref{fig:MultiScaleEpsScaling}).
	\textbf{Center} For large $\veps$, the runtime decreases with $\veps$, since the number of variables decreases (cf.~left plot). For smaller $\veps$, the runtime increases again, since more stages of $\veps$-scaling are required. %
	\textbf{Right} The optimal regularized dual variables were transformed into feasible unregularized dual variables, by decreasing each $\alpha(x)$ until all dual constraints $\alpha(x)+\beta(y) \leq c(x,y)$ were met, \eqref{eqn:OptimalTransportDual}. The sub-optimality of these dual variables is shown. As expected (see Sect.~\ref{sec:IntroductionRelatedWork}) they converge towards a dual optimizer. Absolute optimal value was between $100$ and $400$ for the used test problems, i.e.~for small $\veps$, sub-optimality is small compared to total scale. %
	}
	\label{fig:expEpsComparison}
\end{figure}

The impact of different final values for $\veps$ is outlined in Fig.~\ref{fig:expEpsComparison}.
As expected, the number of variables in the truncated kernel increases with $\veps$.
This leads to two competing trends in the overall runtime: For large $\veps$, the kernel truncation is less efficient, leading to an increase with $\veps$. For small $\veps$, the number of variables is very small, but more and more stages of $\veps$-scaling are necessary, increasing the runtime as $\veps$ decreases further.
Convergence of the regularized optimal dual variables to the unregularized optimal duals is exemplified in the right panel, justifying the use of the approximate entropy regularization technique for transport-type problems.
While one may consider the dual sub-optimality at $\veps \approx 30\,h^2$ sufficiently accurate, we point out that the corresponding primal coupling still contains considerable blur (cf.~Fig.~\ref{fig:MultiScaleEpsScaling}) and that due to less sparsity the runtime is actually higher than for $\veps \approx h^2$.

As illustrated by Figs.~\ref{fig:expImageBenchmark} and \ref{fig:expEpsComparison}, by choosing the threshold for the stopping criterion and the desired final $\veps$, one can tune between required precision and available runtime.

\begin{remark}[Interplay of Modifications]
	\label{rem:EnhancementRelation}
	The numerical findings presented in Figs.~\ref{fig:expEnhancements}-\ref{fig:expEpsComparison} underline how each of the modifications discussed in Sect.~\ref{sec:Algorithm} builds on the previous ones and that all four of them are required for an efficient algorithm.
	The log-domain stabilization is an indispensable prerequisite for running the scaling algorithms with small regularization.
	However, for small $\veps$, convergence tends to become extremely slow (cf.~Fig.~\ref{fig:expEnhancements}), therefore $\veps$-scaling is needed to reduce the number of iterations.
	For small $\veps$, kernel truncation significantly reduces the number of variables and accelerates the algorithm (cf.~Figs.~\ref{fig:expEnhancements} and \ref{fig:expEpsComparison}).
	However, for large $\veps$ (which must be passed during $\veps$-scaling), far fewer variables are truncated and the algorithm cannot be run on large problems.
	This can be avoided by using the coarse-to-fine scheme, completing the algorithm.
	In principle it is possible, only to combine log-domain stabilization with kernel truncation, and to skip $\veps$-scaling and the coarse-to-fine scheme. While this tends to solve the stability and memory issues, convergence is still impractically slow.
\end{remark}

\subsection{Versatility}
\label{sec:NumericsVersatility}
The framework of scaling algorithms developed in \cite{ChizatEntropicNumeric2018}, see Sect.~\ref{sec:Background}, allows to solve more general transport-type problems for which the enhancements of Sect.~\ref{sec:Algorithm} still apply.
We now give some examples to demonstrate this flexibility.
The scope of the following examples is similar to \cite{ChizatEntropicNumeric2018}, but with Algorithm \ref{alg:Full} one can solve larger problems with smaller regularization.

\begin{figure}
	\centering
	\includegraphics{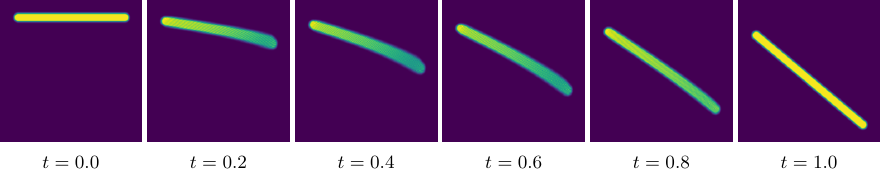}
	\caption{%
	Geodesic for Wasserstein-Fisher-Rao distance on $[0,1]^2$, approximated by a $256 \times 256$ grid, computed as barycenters between end-points with varying weights. %
	Mass that travels further, is decreased during transport to save cost.%
	}
	\label{fig:versatWFGeodesic}
\end{figure}

\paragraph{$\KL$ Fidelity and Wasserstein-Fisher-Rao distance}
For the marginal function $F_X(\sigma) = \lambda \cdot \KL_X(\sigma|\mu)$ with a given reference measure $\mu \in \R_+^X$ and a weight $\lambda > 0$, see Def.~\ref{def:UnbalancedTransport}, one obtains for the (stabilized) $\proxdivSymb$ operator
\begin{align}
	\label{eqn:KLProxDiv}
	\proxdiv{F_X}{\veps}(\sigma) & = \left( \mu \oslash \nu \right)^{\frac{\lambda}{\lambda+\veps}}, &
	\proxdiv{F_X}{\veps}(\sigma,\alpha) & =
		\exp\left(- \tfrac{\alpha}{\lambda + \veps} \right) \odot
			\left( \mu \oslash \nu \right)^{\frac{\lambda}{\lambda+\veps}}.
\end{align}
A proof is given in \cite{ChizatEntropicNumeric2018}.
Compared to the standard Sinkhorn algorithm, the only modification is the pointwise power of the iterates.
As $\lambda \to \infty$ the Sinkhorn iterations are recovered.
In the stabilized operator only the exponential $\exp\big(\tfrac{-\alpha}{\lambda + \veps}\big)$ needs to be evaluated, which remains bounded as $\veps \to 0$.
Algorithm \ref{alg:Full} performs similarly with $\KL$-fidelity as with fixed marginal constraints, allowing to efficiently solve large unbalanced transport problems. Since the truncation scheme can also be used with non-standard cost functions such as \eqref{eqn:WFRCost}, this includes in particular the Wasserstein-Fisher-Rao (WFR) distance.
Fig.~\ref{fig:versatWFGeodesic} shows a geodesic for the WFR distance, to intuitively illustrate its properties.
The geodesic has been computed as weighted barycenters between its endpoints (see below). For a direct dynamic formulation we refer to \cite{KMV-OTFisherRao-2015,ChizatOTFR2015,LieroMielkeSavare-HellingerKantorovich-2015a}. For the relation to the $\KL$ soft-marginal formulation, Def.~\ref{def:UnbalancedTransport}, see \cite{LieroMielkeSavare-HellingerKantorovich-2015a,ChizatDynamicStatic2018}.

\paragraph{Wasserstein barycenters}
Wasserstein barycenters as a natural generalization of the Riemannian center of mass have been studied in \cite{WassersteinBarycenter}.
The computation of entropy regularized Wasserstein barycenters with a Sinkhorn-type scaling algorithm has been presented in \cite{BenamouIterativeBregman2015}, an alternative numerical approach can be found in \cite{Cuturi14Barycenters}.
The iterations can be considered as a special case of the framework in \cite{ChizatEntropicNumeric2018}.
Here, we very briefly recall the iterations. Derivations and proofs can be found in \cite{BenamouIterativeBregman2015}.

\begin{figure}
	\centering
	\includegraphics{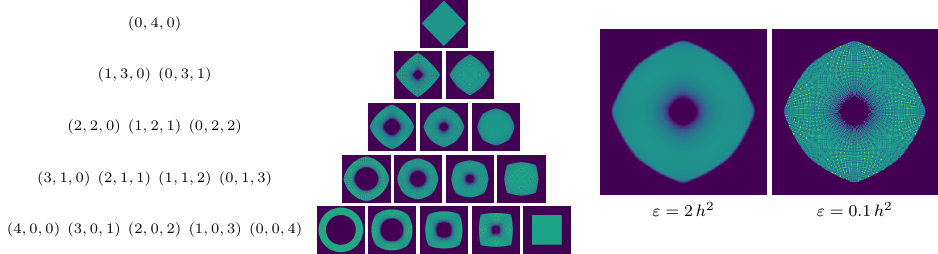}
	\caption{%
	Barycenters in Wasserstein space over $[0,1]^2$, computed on $256 \times 256$ grids for $\veps=0.1\,h^2$.
	\textbf{Left} Weights $4 \cdot (\lambda_1,\lambda_2,\lambda_3)$ for shown barycenters. %
	\textbf{Center} `Barycentric triangle' spanned by a ring, a diamond and a square %
		for weights on the left. %
	\textbf{Right} Close-up of the $\lambda=(1,2,1)/4$ barycenter for $\veps=2\,h^2$ (as reported in \cite{BenamouIterativeBregman2015}) and $\veps=0.1\,h^2$, computed with the adapted algorithm. The $\veps=0.1\,h^2$ version is much sharper, revealing discretization artifacts. %
	}
	\label{fig:versatBarycenter}
\end{figure}

We want to compute the (entropy regularized) Wasserstein barycenter of a tuple $(\mu_1,\ldots,\mu_n) \allowbreak \in \R^{X \times n}$ over a common base space $X=Y$ with metric $d$ with non-negative weights $(\lambda_1,\ldots,\lambda_n)$ that sum to one.
The primal functional can be written as an optimization problem over a tuple $(\pi_i)_{i=1}^n=(\pi_1,\ldots,\pi_n) \in \R^{(X \times X) \times n}$ of couplings, which requires a slight generalization of Def.~\ref{def:RegularizedGenericFormulation}, see \cite{ChizatEntropicNumeric2018}. It is given by
\begin{align}
	\label{eqn:BarycenterPrimal}
	E((\pi_i)_i) & = F_1((\projx \, \pi_i)_i)
		+ F_2((\projy\,\pi_i)_i) + \sum_{i=1}^n \lambda_i\,\KL(\pi_i|\kernel)
\end{align}
where
\begin{align*}
	F_1((\nu_i)_i) & = \sum_{i=1}^n \iota_{\{\mu_i\}}(\nu_i), &
	F_2((\nu_i)_i) & = \begin{cases} 0 & \tn{if } \exists\, \sigma \in \R^X \tn{ s.t.~}
		[\sigma=\nu_i \,\forall\, i=1,\ldots,n], \\
		+\infty & \tn{else.}
		\end{cases}
\end{align*}
and $\kernel$ is the kernel \eqref{eqn:Kernel} over $X \times X$ for the cost $c=d^2$.
When an optimizer $(\pi^\dagger_i)_i$ is found, the common second marginal of all $\pi^\dagger_i$ is the sought-after barycenter.
To solve \eqref{eqn:BarycenterPrimal} one considers again a suitable dual problem and uses alternating optimization. Updates corresponding to $F_1$ decompose into independent standard Sinkhorn iterations for each marginal, the update for $F_2$ couples all marginals, see \cite{BenamouIterativeBregman2015,ChizatEntropicNumeric2018}. The adaptations from Sect.~\ref{sec:Algorithm} remain applicable.
A barycentric triangle computed with Algorithm \ref{alg:Full} is shown in Fig.~\ref{fig:versatBarycenter}.
The log-domain stabilization allows to reach a lower final regularization $\veps$ as for example in \cite{BenamouIterativeBregman2015}.
Regularization can be made so small that discretization artifacts become visible. While this may not look entirely pleasing, it clearly gives a better approximation to the unregularized problem and illustrates that with log-domain stabilization entropy regularized numerical methods can produce sharp results.

\paragraph{Wasserstein-Fisher-Rao barycenters}
Similarly one can define barycenters for transport distances with $\KL$ marginal fidelity, which includes the Gaussian Hellinger-Kantorovich (GHK) distance and the Wasserstein-Fisher-Rao (WFR) distance (Def.~\ref{def:UnbalancedTransport}).
The primal functional is given by \eqref{eqn:BarycenterPrimal} with
\begin{align*}
	F_1((\nu_i)_i) & = \Lambda \cdot \sum_{i=1}^n \lambda_i \, \KL(\nu_i|\mu_i)\,, &
	F_2((\nu_i)_i) & = \inf_{\sigma \in \R^X} \Lambda \cdot \sum_{i=1}^n \lambda_i \, \KL(\nu_i|\sigma)\,,
\end{align*}
where $\Lambda>0$ is a global weight of the $\KL$-fidelity. When a primal optimizer is found, the minimizing $\sigma$ in $F_2$ yields the sought-after barycenter. We refer to \cite{ChizatEntropicNumeric2018} for details.
Partial optimization corresponding to $F_1$ can again be done separately for each marginal, leading to $\KL$ fidelity updates as given by \eqref{eqn:KLProxDiv}. The update corresponding to $F_2$ is again coupled \cite{ChizatEntropicNumeric2018}, adaptations from Sect.~\ref{sec:Algorithm} remain applicable.

\begin{figure}
	\centering
	\includegraphics{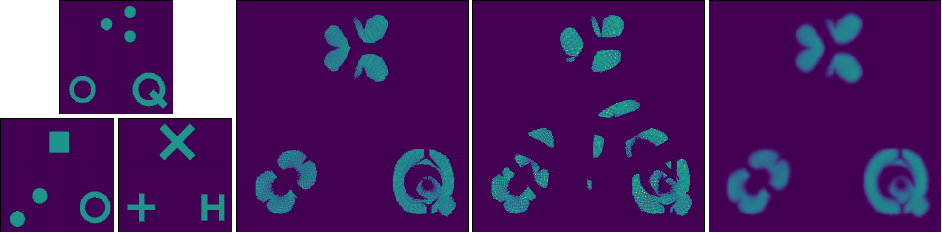}
	\caption{Comparison of different barycenters models: %
		\textbf{First column} Corner points of barycentric triangle. Each measure consists of three `groups'.
		\textbf{Second column} Wasserstein-Fisher-Rao barycenter for $\lambda=(1,2,1)/4$ for $\veps=0.1\,h^2$. %
		\textbf{Third column} Wasserstein barycenter between normalized reference measures for $\veps=0.1\,h^2$. Unlike the `unbalanced' barycenters, here mass must be transferred between the different `groups' of the reference measures. %
		\textbf{Fourth column} Gaussian Hellinger-Kantorovich barycenter for $\veps=6.55\,h^2$, as computed with Gaussian convolution without log-domain stabilization. %
	}
	\label{fig:versatBarycenterWFCloseup}
\end{figure}

\paragraph{Wasserstein Gradient Flows}
\begin{figure}
	\centering %
	\hfill %
	\includegraphics[height=5cm]{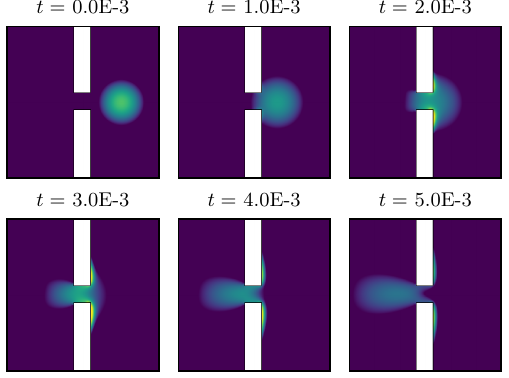} %
	\hfill %
	\includegraphics[height=5cm]{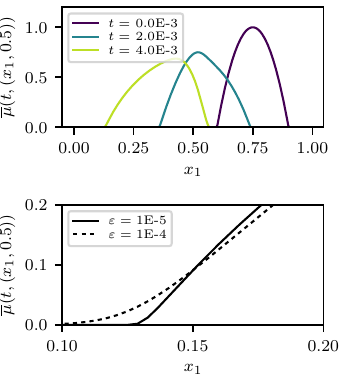} %
	\hfill
	\caption{\textbf{Left} Entropic Wasserstein gradient flow for the porous media equation on $[0,1]^2$, approximated by a $256 \times 256$ grid with $\veps=10^{-5} \approx 0.66\,h^2$, $\tau=2 \cdot 10^{-4}$. %
	The energy is given by \eqref{eqn:GradientFlowEnergy} with $v((x_1,x_2)) = 100 \cdot x_1$ if $x=(x_1,x_2) \in \Omega$, $v(x)=+\infty$ otherwise and $\Omega = [0,1]^2 \setminus \hat{\Omega}$ where $\hat{\Omega}$ is a `barrier' indicated by the white rectangles. %
	\textbf{Top Right} Cross section of density at different times along $x_2 = 0.5$. %
	\textbf{Bottom Right} Close-up for $t=4 \cdot 10^{-3}$ for different values of regularization $\veps$. %
	For $\veps=10^{-5}$ the compact support of $\ol{\mu}$, a characteristic feature of the porous media equation, is numerically well preserved.
	Without log-domain stabilization, for $\veps=10^{-4}$ the entropic blur quickly distorts this feature.
	}
	\label{fig:versatGradientFlow}
\end{figure}
In \cite{PeyreEntropicFlows2015} diagonal scaling algorithms were extended to compute proximal steps for entropy regularized optimal transport to approximate gradient flows in Wasserstein space (cf.~Sect.~\ref{sec:IntroductionRelatedWork}). This was then subsumed into the general framework of \cite{ChizatEntropicNumeric2018}.
Here we given an example for the porous medium equation, for more details we refer to \cite{PeyreEntropicFlows2015,ChizatEntropicNumeric2018}. Let
\newcommand{\leb}{\mc{L}}
\newcommand{\muInt}{\ol{\mu}}
\newcommand{\LambertW}{\mc{W}}
\begin{align}
	\label{eqn:GradientFlowEnergy}
	F & : \prob(X) \to \RExt, &
	\mu & \mapsto \sum_{x \in X} u\left(\tfrac{\mu(x)}{\leb(x)}\right) \, \leb(x)
		+ \sum_{x \in X} v(x)\,\mu(x)
\end{align}
where $u(s) = s^2$, $\leb$ is the discretized Lebesgue measure on $X \subset \R^d$ and $v : X \to \RExt$ is a potential.
Then, for some initial $\iterz{\mu} \in \prob(X)$ and a time step size $\tau > 0$ we iteratively construct a sequence $(\iterl{\mu})_\ell$ where $\iterll{\mu}$ is given by the proximal step of $F$ with step size $\tau$ w.r.t.~the entropy regularized Wasserstein distance on $X$ from reference point $\iterl{\mu}$.
Based on Def.~\ref{def:EntropicOptimalTransport}, $\iterll{\mu}$ can be computed as follows:
\begin{align}
	\label{eqn:GradientFlowPrimal}
	\iterll{\pi} & \eqdef \argmin_{\pi \in \prob(X^2)}
		\left( \iota_{\{\iterl{\mu}\}}(\projx\,\pi) + 2\,\tau \cdot F(\projy\,\pi)
		 + \veps\,\KL(\pi|\kernel) \right), &
	\iterll{\mu} & \eqdef \projy \iterll{\pi}\,,
\end{align}
where $\kernel$ is the kernel w.r.t.~the squared Euclidean distance on $X$.
Then introduce the time-continuous interpolation $\muInt : \R_+ \to \prob(X)$, $t \mapsto \iterl{\mu}$ when $t \in [\tau \cdot \ell, \tau \cdot (\ell +1))$.
Consider now the limit $(\tau,\veps) \to 0$ in a way such that $\veps |\log \veps| \leq \tau^2$. Then, up to discretization, the function $\muInt$ converges to a solution of the porous media PDE $\partial_t \muInt = \Delta (\muInt^2) + \ddiv( \muInt \cdot \nabla v)$.
A proof is given in \cite{Carlier-EntropyJKO-2015}.
Problem \eqref{eqn:GradientFlowPrimal} is an instance of Def.~\ref{def:RegularizedGenericFormulation} and can be solved by alternating dual optimization \cite{ChizatEntropicNumeric2018}.
A numerical example is shown in Fig.~\ref{fig:versatGradientFlow}.
As in the previous experiments, Algorithm \ref{alg:Full} allows to use log-domain stabilization on large problems, producing sharp results. In this example, the compact support of the porous media equation is numerically well preserved.


\section{Conclusion}
\label{sec:Conclusion}
Scaling algorithms for entropy regularized transport-type problems have become a wide-spread numerical tool.
Naive implementations have some severe numerical limitations, in particular for small regularization and on large problems.
In this article, we proposed an enhanced variant of the standard scaling algorithm to address these issues:
Diverging scaling factors and slow convergence are remedied by log-domain stabilization and $\veps$-scaling. Required runtime and memory are significantly reduced by adaptive kernel truncation and a coarse-to-fine scheme.
A new convergence analysis for the Sinkhorn algorithm was developed.
Numerical examples showed the efficiency of the enhanced algorithm, confirmed the scaling predicted by the convergence analysis and demonstrated that the algorithm can produce sharp results on a wide range of transport-type problems.
Potential directions for future research are the more detailed study of $\veps$-scaling, a more systematic understanding of the stability of the log-domain stabilization and application to multi-marginal problems.

\paragraph{Acknowledgements}
L\'ena\"ic Chizat, Luca Nenna and Gabriel Peyr\'e are thanked for stimulating discussions. Bernhard Schmitzer was supported by the European Research Council (project SIGMA-Vision).

\appendix

\section{Additional Proofs}

\label{sec:ApxAdditionalProofs}

\subsection%
	[Proof of Lemma \ref*{lem:EffectiveDiagonalProblemStability}]%
	{Proof of Lemma	\ref{lem:EffectiveDiagonalProblemStability}}
	First, we establish existence of minimizers. For some $\veps>0$, $d \in \R^{\nK \times \nK}$ the functional $\beta \mapsto \Jdiag_{\veps,d}(\beta)$ is convex and bounded from below. Further, it is invariant under adding the same constant to all components of $\beta$. Hence, the optimal value $\min_\beta \Jdiag_{\veps,d}(\beta)$ is not changed by adding the constraint $\beta(1)=0$.
	With this added constraint the functional becomes strictly convex and coercive in the remaining variables, hence a unique minimizer exists. The full set of minimizers is then obtained via constant shifts.
		
	The first order optimality condition for the functional yields for the $i$-th component of $\beta$:
	\begin{align*}
		\beta(i) = \frac{1}{2} \left[ \softmax_{j: j\neq i}(-d(i,j)+\beta(j),\veps) + \softmin_{j: j\neq i}(d(j,i)+\beta(j),\veps) \right],
	\end{align*}
	where the subscript $j : j \neq i$ denotes that $\softmax$ is taken only over components $\{1,\ldots,\nK\} \setminus \{i\}$.
	Finiteness of $d$ ensures that this expression is meaningful.
	Let $i_1 \in \{1,\ldots,\nK\}$ be an index where $\Delta \beta$ is maximal, i.e.\ $\Delta \beta(i_1) = \max \Delta \beta$.
	
	From the optimality conditions for $\beta_a(i_1)$, $a=1,2$, and \eqref{eqn:Softbounds} we obtain:
	\begin{align*}
		\beta_a^\dagger(i_1) & = \frac{1}{2} \left[ \softmax_{j : j \neq i_1}(
				-d_a(i_1,j) + \beta_a^\dagger(j), \veps_a)
			+ \softmin_{j : j \neq i_1}(
				d_a(j,i_1) + \beta_a^\dagger(j), \veps_a) \right], \\
		\Delta \beta(i_1) & \leq \frac{1}{2} \left[
			\max_{j:j \neq i_1} (-\Delta d(i_1,j) + \Delta \beta(j)) +
			\max_{j:j \neq i_1} ( \Delta d(j,i_1) + \Delta \beta(j)) + (\veps_1 + \veps_2) \cdot \log \nK \right] \\
			& \leq \max_{j:j \neq i_1} ( w(i_1,j) + \Delta \beta(j)) + \veps_1\, \log \nK
	\end{align*}
	where $w(i,j) = \max\{ -\Delta d(i,j), \Delta d(j,i) \}$. This implies there is some $i_2 \in \{1,\ldots,\nK\} \setminus \{i_1\}$ with
	\begin{align*}
		\Delta \beta(i_2) \geq \Delta \beta(i_1) - w(i_1,i_2) - \veps_1\, \cdot \log \nK\,.
	\end{align*}
	
	We will call the index $i_2$ a child of $i_1$.
	We now repeat this reasoning to derive lower bounds for other entries of $\Delta \beta$. For this we must `remove' the index $i_2$ from the problem, defining a reduced problem.
	Let $I_1 = \{i_1,i_2\}$ and let $I_2 = \{1,\ldots,\nK\} \setminus I_1$.
	We will keep all variables of $\beta$ with indices in $I_2$, but describe all variables with indices in $I_1$ by a single reduced variable. For this we consider vectors in $\R^{1+|I_2|}$, where we index the entries by $\{i_1\} \cup I_2$. One can think of this as a vector in $\R^\nK$, where we have `crossed out' entries corresponding to $I_1$ and replaced them by a single effective entry, indexed with $i_1$.
	For $a=1,2$ we consider the reduced functionals $\JdiagEff_a : \hat{\beta} \mapsto \Jdiag_{\veps_a,d_a}(\tilde{\beta}_a + B\,\hat{\beta})$	where $\tilde{\beta}_a \in \R^\nK$ is a constant offset, $\hat{\beta} \in \R^{1 + |I_2|}$ is the reduced variable and $B \in \R^{\nK \times (1+|I_2|)}$ is a matrix that implements the parametrization. We set
	\begin{align*}
		\tilde{\beta}_a(j) & = \begin{cases}
			\beta_a^\dagger(j) - \beta_a^\dagger(i_1) & \tn{if } j \in I_1, \\
			0 & \tn{else,}
			\end{cases}
		&
		B(j,k) & = \begin{cases}
			1 & \tn{if } j \in I_1,\, k=i_1, \\
			1 & \tn{if } j=k \in I_2, \\
			0 & \tn{else.}
			\end{cases}
	\end{align*}
	So the reduced functionals are given by
	\begin{align*}
		\JdiagEff_a(\hat{\beta}) & = \sum_{\substack{j \in I_1,\\ k \in I_1}}
			\exp([-d_a(j,k) - \tilde{\beta}_a(j) + \tilde{\beta}_a(k)]/\veps_a)
			+ \sum_{\substack{j \in I_1,\\ k \in I_2}}
				\exp([-d_a(j,k) - \tilde{\beta}_a(j) -\hat{\beta}(i_1) + \hat{\beta}(k)]/\veps_a) \\
			& \qquad + \sum_{\substack{j \in I_2,\\ k \in I_1}}
				\exp([-d_a(j,k) - \hat{\beta}(j) + \tilde{\beta}_a(k) + \hat{\beta}(i_1)]/\veps_a)
			+ \sum_{\substack{j \in I_2,\\ k \in I_2}}
				\exp([-d_a(j,k) - \hat{\beta}(j) + \hat{\beta}(k)]/\veps_a) \\
		& = \sum_{\substack{j \in \{i_1\} \cup I_2,\\ k \in \{i_1\} \cup I_2}}
			\exp([-D_a(j,k) - \hat{\beta}(j) + \hat{\beta}(k)]/\veps_a)
	\end{align*}
	with the reduced coefficient matrix $D_a \in \R^{(1+|I_2|)^2}$ with entries
	\begin{align*}
		D_a(j,k) & = \begin{cases}
				\softmin_{r \in I_1,\, s \in I_1}
				\left( d_a(r,s) + \tilde{\beta}_a(r) - \tilde{\beta}_a(s), \veps_a \right) & \tn{if } j = i_1,\, k = i_1, \\
				\softmin_{r \in I_1} \left(
				d_a(r,k) + \tilde{\beta}_a(r), \veps_a \right) & \tn{if } j=i_1,\, k\in I_2, \\
				\softmin_{s \in I_1} \left(
				d_a(j,s) - \tilde{\beta}_a(s), \veps_a \right) & \tn{if } j \in I_2,\, k = i_1, \\
				d_a(j,k) & \tn{if } j \in I_2,\, k \in I_2.
			\end{cases}
	\end{align*}
	Consider the reduced variables $\hat{\beta}_{a}^\dagger \in \R^{1 +|I_2|}$ with entries
	\begin{align*}
		\hat{\beta}^{\dagger}_a(j) = \begin{cases}
			\beta_a^\dagger(i_1) & \tn{if } j=i_1, \\
			\beta_a^\dagger(j) & \tn{if } j \in I_2.
			\end{cases}
	\end{align*}
	Then $\beta_a^\dagger = \tilde{\beta}_a + B\,\hat{\beta}_{a}^{\dagger}$ and therefore $\hat{\beta}_{a}^{\dagger}$ are minimizers of $\JdiagEff_a$. Note also that $\hat{\beta}_{2}^{\dagger}(j)-\hat{\beta}_{1}^{\dagger}(j) = \Delta \beta(j)$ for $j \in \{i_1\} \cup I_2$.
	Using the optimality conditions for the reduced functionals and arguing as above, we find
	\begin{align*}
		\Delta \beta(i_1) & \leq \max_{k \in I_2} (W(i_1,k) + \Delta \beta(k)) + \veps_1\,\log \nK
	\end{align*}
	where $W(i_1,k) = \max\{ -\Delta D(i_1,k), \Delta D(k,i_1) \}$ for $k \in I_2$ and $\Delta D = D_2 - D_1$. With \eqref{eqn:Softbounds} we find
	\begin{align*}
		-\Delta D(i_1,k) & \leq \max_{j \in I_1} (-\Delta d(j,k) - \Delta \beta(j) + \Delta \beta(i_1)) + \veps_2\, \log \nK\,, \\
		\Delta D(k,i_1) & \leq \max_{j \in I_1} (\Delta d(k,j) - \Delta \beta(j) + \Delta \beta(i_1)) + \veps_1 \, \log \nK\,
	\end{align*}
	and eventually $W(i_1,k) \leq \max_{j \in I_1} ( w(j,k) - \Delta \beta(j)) + \Delta \beta(i_1) + \max\{ \veps_1, \veps_2 \} \cdot \log \nK$.
	So there is some index $i_3 \in I_2$ such that
	\begin{align*}
		\Delta \beta(i_3) \geq \min_{j \in I_1} \left( -w(j,i_3) + \Delta \beta(j) \right) - 2\veps_1\,\log \nK\,.
	\end{align*}
	The index $i_3$ will be called a child of the minimizing index $j \in I_1$ on the r.h.s.~(or one of the minimizing indices). Then we add $i_3$ to the set $I_1$ and repeat the argument with the reduced functional, to obtain an index $i_4$ and repeat this until $I_1$ contains all indices.
	
	Since we assign every new index $i_k$ that is added to $I_1$ as a child to one parent node in $I_1$, this also constructs a tree graph with root node $i_1$ (finiteness of $d$ and consequently $D$ implies that this graph is connected). For an index $i_k$ let $(i_1,i_2,\ldots,i_k)$ be the unique path from the root to $i_k$. Then
	\begin{align*}
		\Delta \beta(i_k) & \geq -\sum_{j=2}^k w(i_{j-1},i_j) + \Delta \beta(i_1) - 2\,(k-1)\,\veps_1 \, \log \nK 
		\geq - \maxdiam(w)+ \Delta \beta(i_1) - 2\,\veps_1\,\nK\,\log \nK\,.
	\end{align*}
	Since $\Delta \beta(i_1) = \max \Delta \beta$ the result follows.

\subsection%
	[Proof of Theorem \ref*{thm:AuctionEpsStability}]%
	{Proof of Theorem \ref{thm:AuctionEpsStability}}

		Let $\pi_1$, $\pi_2$ be the primal optimizers associated with $(\alpha_1,\beta_1)$ and $(\alpha_2,\beta_2)$ and consider the assignment graph for $\pi_1$ and $\pi_2$ and threshold $1/M$ (see Lemma \ref{lem:AssignmentGraph}).
		Let $\{(X_i,Y_i)\}_{i=1}^\nK$ be the strongly connected components of the assignment graph. By virtue of Lemma \ref{lem:AssignmentGraph}\eqref{item:AssignGraphCycleBalanced}, $\mu(X_i) = \nu(Y_i)$ for $i = 1,\ldots,\nK$.
		Pick some representatives $\{y_i\}_{i=1}^\nK \subset Y$ such that $y_i \in Y_i$ for $i=1,\ldots,\nK$.
		
		For $a=1,2$, let now $\Jdiag_a$ be the reduced effective diagonal functionals, defined in Lemma \ref{lem:EffectiveDiagonalProblem}, corresponding to spaces $(X,Y)$, marginals $(\mu,\nu)$, parameters $\veps_a$, cost $c$, the partitions given by the strongly connected components and the representatives $\{y_i\}_{i=1}^\nK$.
		Let $d_a$ be the corresponding effective coefficients (finite, since $c$ is finite), let $\hat{\beta}^\dagger_{a}$ be two corresponding maximizers and let $\Delta d = d_2 - d_1$, $\Delta \hat{\beta} = \hat{\beta}^\dagger_2 - \hat{\beta}^\dagger_1$.		
		By virtue of Lemma \ref{lem:EffectiveDiagonalProblemStability} one has $\max \Delta \hat{\beta} - \min \Delta \hat{\beta} \leq \maxdiam(w) + 2\,\veps^1\,\nK\,\log \nK$, where $w \in \R^{\nK \times \nK}$ with $w(i,j) = \max\{-\Delta d(i,j),\Delta d(j,j) \}$.

	Now we derive some estimates on $\Delta d$. Consider once more the assignment graph for $\pi_1$, $\pi_2$ and threshold $1/M$.
	For every edge $y \to x$ we have (using \eqref{eqn:EntropyPDRelation})
	\begin{align*}
		\alpha_1(x) + \beta_1(y) - c(x,y) \geq -\veps_1 \, \log M\,.
	\end{align*}
	Moreover, from the marginal conditions we find $\pi_2(x,y) \leq \nu(y)$, which implies
	\begin{align*}
		\alpha_2(x) + \beta_2(y) - c(x,y) \leq \veps_2\,\log M.
	\end{align*}
	Combining the two estimates, we obtain $\Delta \alpha(x) + \Delta \beta(y) \leq (\veps_1 + \veps_2) \, \log M \leq 2\veps_1\,\log M \eqdef L$.
	Similarly, for edges $x \to y$ we obtain $\Delta \alpha(x) + \Delta \beta(y) \geq -(\veps_1 + \veps_2) \, \log M \geq - L$.
	Let now $(y_1,x_1,\ldots,y_k)$ be an alternating path in $(X,Y)$, then, by combining the above inequalities we find $\Delta \beta(y_{j+1}) \geq \Delta \beta(y_j) - 2 \cdot L$ for $j=1,\ldots,k-1$ and eventually
	\begin{align*}
		\Delta \beta(y_{k})-\Delta \beta(y_1) & \geq  - 2 \cdot (k-1) \cdot L\,.
	\end{align*}
	Similarly, for a path $(x_1,y_2,x_2,\ldots,y_k)$ get $\Delta \alpha(x_{1})+\Delta \beta(y_k) \geq  - (2\,k-1) \cdot L$, and for a path $(y_1,x_1,\ldots,y_k,\allowbreak x_k)$ get $\Delta \alpha(x_{k})+\Delta \beta(y_1) \leq  (2\,k-1) \cdot L$.
	
	Consider now a partition cell $(X_i,Y_i)$ and let $y_i \in Y_i$ be the selected `representative', as described above. For every $y \in Y_i$ there is a path to and from $y_i$ with at most $2(|Y_i|-1)$ edges, for every $x \in X_i$ there is a path to and from $y_i$ with at most $2\,|Y_i|-1$ edges. With $\Delta \tilde{\alpha}(x)=\Delta \alpha(x) + \Delta \beta(y_i)$ and $\Delta \tilde{\beta}(y) = \Delta \beta(y) - \Delta \beta(y_i)$ we therefore obtain
	\begin{align*}
		|\Delta \tilde{\alpha}(x)| & \leq (2\,|Y_i|-1) \cdot L, &
		|\Delta \tilde{\beta}(y)| & \leq 2\,(|Y_i|-1) \cdot L.
	\end{align*}
	We recall \eqref{eqn:EffectiveDiagonalProblemCoefs}
	\begin{align*}
		d_a(i,j) & = \softmin_{\substack{x \in X_i,\\y \in Y_j}}
			\left(
				c(x,y) - \tilde{\alpha}_a(x) - \tilde{\beta}_a(y) - \veps_a\,\log(\mu(x)\,\nu(y)),
				\veps_a
			\right)
	\end{align*}
	and get
	\begin{align*}
		\Delta d(i,j) & \leq \max_{\substack{x \in X_i,\\ y \in Y_j}}
			\left(
				-\Delta \tilde{\alpha}(x) - \Delta \tilde{\beta}(y)
				- \Delta \veps \cdot \log\big(\mu(x)\,\nu(y)\big) \right)
			+ \veps_1 \cdot \log\big(|X_i|\,|Y_j|\big) \\
		& \leq 4\,|Y_i|\,L + \veps_1 \cdot \log\big(|X_i|\,|Y_j|\big) \\
		\Delta d(i,j) & \geq \min_{\substack{x \in X_i,\\ y \in Y_j}}
			\left(
				-\Delta \tilde{\alpha}(x) - \Delta \tilde{\beta}(y)
				- \Delta \veps \cdot \log\big(\mu(x)\,\nu(y)\big) \right)
			- \veps_2 \cdot \log\big(|X_i|\,|Y_j|\big) \\
		& \geq -4\,|Y_i|\,L - \veps_2 \cdot \log\big(|X_i|\,|Y_j|\big)
	\end{align*}
	where we used $|\Delta \veps \log(\mu(x)\,\nu(y))| \leq 2\veps_1\,\log M=L$.
	From this follows that $w(i,j) \leq 8\,\max\{|Y_i|,|Y_j|\}\cdot \veps_1\,\log M+2\,\veps_1\,\log N$, which in turn implies that $\maxdiam(w) \leq 16\,\veps_1\,N\,\log M + 2\,\veps_1\,\nK\,\log N$.
	
		Recall that $\Delta \hat{\beta} = \hat{\beta}_{2}^{\dagger} - \hat{\beta}_{1}^{\dagger}$, where $\hat{\beta}_{a}^{\dagger}$, $a=1,2$, are the optimizers of the effective diagonal problems. Then from Lemma \ref{lem:EffectiveDiagonalProblemStability}, and by bounding $\nK \leq N$ we obtain that
		\begin{align*}
			\max \Delta \hat{\beta} - \min \Delta \hat{\beta} \leq \veps_1\,N\,(4 \log N + 16 \log M)
		\end{align*}
		and finally with $\max \Delta \beta - \min \Delta \beta \leq \max \Delta \tilde{\beta} - \min \Delta \tilde{\beta} + \max \Delta \hat{\beta} - \min \Delta \hat{\beta}$ we get
		\begin{align*}
			\max \Delta \beta - \min \Delta \beta & \leq \veps_1\,N\,(4 \log N + 24 \log M),
		\end{align*}
		and analogously we get the equivalent bound for $\Delta \alpha$.

\bibliography{references}{}
\bibliographystyle{siamplain}

\end{document}